\documentclass[a4paper]{amsart}

\usepackage{amsmath,amssymb,amsthm}
\usepackage{fullpage}
\usepackage{hyperref}
\usepackage{mathtools}
\usepackage{graphicx}
\usepackage{xypic}
\entrymodifiers={+!!<0pt,\fontdimen22\textfont2>}
\usepackage[all]{xy}
\usepackage{tikz}
\usetikzlibrary{arrows,decorations.markings}

\newtheoremstyle{myremark} 
    {7pt}                    
    {7pt}                    
    {}  	                 
    {}                           
    {\bf}       	         
    {.}                          
    {.5em}                       
    {}  

\theoremstyle{plain}
\newtheorem{lemma}{Lemma}[section]
\newtheorem{theorem}[lemma]{Theorem}
\newtheorem{corollary}[lemma]{Corollary}
\newtheorem{proposition}[lemma]{Proposition}
\theoremstyle{definition}
\newtheorem{definition}[lemma]{Definition}
\newtheorem*{theorem-main}{Theorem~\ref{thm:ellipse-homotopy}}
\newtheorem*{theorem-secondary}{Theorem~\ref{thm:exactly-m-orbits}}
\theoremstyle{myremark}
\newtheorem{remark}[lemma]{Remark}

\newcommand{\N}{\mathbb{N}}
\newcommand{\R}{\mathbb{R}}
\newcommand{\Z}{\mathbb{Z}}
\newcommand{\F}{\mathbb{F}}

\newcommand{\redhom}{\widetilde{H}}

\newcommand{\cl}{\mathrm{Cl}}
\newcommand{\colim}{\mathrm{colim}}
\newcommand{\hocolim}{\mathrm{hocolim}}

\newcommand{\diredge}{\rightarrow}
\newcommand{\vrless}[2]{\mathrm{VR}_<(#1;#2)}
\newcommand{\vrleq}[2]{\mathrm{VR}_\leq(#1;#2)}
\newcommand{\vr}[2]{\mathrm{VR}(#1;#2)}
\newcommand{\vrcless}[2]{\mathbf{VR}_<(#1;#2)}
\newcommand{\vrcleq}[2]{\mathbf{VR}_\leq(#1;#2)}
\newcommand{\vrc}[2]{\mathbf{VR}(#1;#2)}
\newcommand{\wf}{\mathrm{wf}}
\newcommand{\cnk}[2]{C_{#1}^{#2}}
\newcommand{\ddist}{\vec{d}}
\newcommand{\olen}{\ell}
\newcommand{\owind}{\omega}
\newcommand{\numorb}{P}
\newcommand{\hit}{\mathcal{H}}
\newcommand{\fin}{\mathrm{Fin}}
\newcommand{\tfin}{\widetilde{\mathrm{Fin}}}

\newcommand{\rank}{\mathrm{rank}}
\newcommand{\nullity}{\mathrm{nullity}}
\newcommand{\md}[1]{\ \mathrm{mod}\ #1}

\begin{document}
\title{On Vietoris--Rips complexes of ellipses}
\author{Micha{\l} Adamaszek}
\address{Department of Mathematical Sciences, University of Copenhagen, Universitetsparken 5, 2100 Copenhagen, Denmark}
\email{aszek@mimuw.edu.pl}
\author{Henry Adams}
\address{Department of Mathematics, Colorado State University, Fort Collins, CO 80523, United States}
\email{adams@math.colostate.edu}
\author{Samadwara Reddy}
\address{Department of Mathematics, Duke University, Durham, NC 27708, United States}
\email{samadwara.reddy@duke.edu}
\thanks{MA was supported by VILLUM FONDEN through the network for Experimental Mathematics in Number Theory, Operator Algebras, and Topology. HA was supported in part by Duke University and by the Institute for Mathematics and its Applications with funds provided by the National Science Foundation. The research of MA and HA was supported through the program ``Research in Pairs" by the Mathematisches Forschungsinstitut Oberwolfach in 2015. SR was supported by a PRUV Fellowship at Duke University.}

\begin{abstract}
For $X$ a metric space and $r > 0$ a scale parameter, the Vietoris--Rips simplicial complex $\vrcless{X}{r}$ (resp.\ $\vrcleq{X}{r}$) has $X$ as its vertex set, and a finite subset $\sigma\subseteq X$ as a simplex whenever the diameter of $\sigma$ is less than $r$ (resp.\ at most $r$). Though Vietoris--Rips complexes have been studied at small choices of scale by Hausmann and Latschev \cite{Hausmann1995,Latschev2001}, they are not well-understood at larger scale parameters. In this paper we investigate the homotopy types of Vietoris--Rips complexes of ellipses $Y=\{(x,y)\in \R^2~|~(x/a)^2+y^2=1\}$ of small eccentricity, meaning $1 < a \le \sqrt{2}$. Indeed, we show there are constants $r_1 < r_2$ such that for all $r_1 < r < r_2$, we have $\vrcless{Y}{r}\simeq S^2$ and $\vrcleq{Y}{r}\simeq \bigvee^5 S^2$, though only one of the two-spheres in $\vrcleq{Y}{r}$ is \emph{persistent}. Furthermore, we show that for any scale parameter $r_1 < r < r_2$, there are arbitrarily dense subsets of the ellipse such that the Vietoris--Rips complex of the subset is not homotopy equivalent to the Vietoris--Rips complex of the entire ellipse. As our main tool we link these homotopy types to the structure of infinite cyclic graphs.
\end{abstract}
\maketitle

\section{Introduction}

Let $X$ be a metric space and let $r>0$ be a scale parameter. The Vietoris--Rips simplicial complex $\vrcless{X}{r}$ (resp.\ $\vrcleq{X}{r}$) has $X$ as its vertex set and has a finite subset $\sigma\subseteq X$ as a simplex whenever the diameter of $\sigma$ is less than $r$ (resp.\ at most $r$). When we do not want to distinguish between $<$ and $\leq$, we use the notation $\vrc{X}{r}$ to denote either complex. Vietoris--Rips complexes have been used to define a homology theory for metric spaces \cite{Vietoris27}, in geometric group theory \cite{Gromov1987}, and more recently in applied and computational topology \cite{EdelsbrunnerHarer,Carlsson2009}.

In applications of topology to data analysis \cite{Carlsson2009}, metric space $X$ is often a finite dataset sampled from an unknown underlying space $M$, and one would like to use $X$ to recover information about $M$. A common technique is to use the complex $\vrc{X}{r}$ as an approximation for $M$. The idea of \emph{persistent homology} is to track the homology of $\vrc{X}{r}$ as the scale $r$ increases, to trust homological features which persist over a long range of scale parameters as representing true features of $M$, and to disregard short-lived homological features as noise. There are several theoretical results which justify this approach:
\begin{itemize}
\item[(i)] If $M$ is a Riemannian manifold and scale parameter $r$ is sufficiently small, then Hausmann proves $\vrcless{M}{r}$ is homotopy equivalent to $M$ \cite[Theorem~3.5]{Hausmann1995}.
\item[(ii)] If $M$ is a Riemannian manifold, scale parameter $r$ is sufficiently small, and metric space $X$ is sufficiently close to $M$ in the Gromov-Hausdorff distance (for example if $X\subseteq M$ is sufficiently dense), then Latschev proves $\vrcless{X}{r}$ is homotopy equivalent to $M$ \cite[Theorem~1.1]{Latschev2001}.
\item[(iii)] If $M\subseteq\R^d$ is compact with sufficiently large weak feature size and finite set $X\subseteq \R^d$ is sufficiently close to $M$ in the Hausdorff distance, then the homology of $M$ can be recovered from the persistent homology of $\vrc{X}{r}$ as the scale $r$ varies over sufficiently small scale parameters \cite[Theorem~3.6]{ChazalOudot2008}.
\item[(iv)] If $X_1,X_2,X_3,\ldots$ are a sequence of noisy samplings near a metric space $M$ that converges to $M$ in the Gromov-Hasudorff distance as $n\to\infty$, then the persistent homology of $\vrc{X_n}{-}$ converges to a limiting object, namely the persistent homology of $\vrc{M}{-}$. This is called the \emph{stability} of persistent homology \cite{chazal2009gromov,ChazalDeSilvaOudot2013}. 
\end{itemize}
Properties (i)-(iii) rely on sufficiently small scale parameters $r$; however, since space $M$ is unknown, one does not typically know what values of $r$ are small enough to satisfy the above theorems. Instead, practitioners allow $r$ to increase from zero to a value which is often larger than the ``sufficiently small" hypotheses in theorems (i)-(iii) above. Hence Vietoris--Rips complexes with large scale parameters frequently arise in applications of persistent homology, even though little is known about their behavior.

Property (iv), by contrast, holds for all scale parameters $r$. It implies that if $X_n$ is a uniformly random sample of $n$ points from a compact manifold $M$, then the sequence of persistent homology diagrams for $\vrc{X_n}{-}$ converges as $n\to\infty$ to a limiting object, the persistent homology of $\vrc{M}{-}$. However, very little is known about the limiting object of this convergent sequence.\footnote{By analogy, it is as if we have a Cauchy sequence in a complete metric space---we know that the Cauchy sequence converges to some limit point, but we don't know the limit point.} For many manifolds $M$ it is not known at what scale $r$ the homotopy equivalence $\vrc{M}{r}\simeq M$ (guaranteed by (i) for small $r$) ends, or what homotopy type occurs next. In particular, a data analyst should not immediately conclude that their data has a two-dimensional hole if its Vietoris--Rips complex has a two-dimensional persistence interval; as we will show, any sufficiently dense sampling of the ellipse also has a two-dimensional persistent homology interval of nonzero length. The inferences one can make about $M$ from estimates of the persistent homology of $\vr{M}{r}$ would be aided if we knew the persistent homology of $\vr{M}{r}$ for a wider classes of shapes $M$. Examples are needed to aid the development of a more general theory, and in this paper we develop tools to address the example of the ellipse.

To our knowledge the circle is the only\footnote{And easy consequences thereof, such as $n$-dimensional tori equipped with the $L^\infty$ metric \cite[Section~10]{AA-VRS1}.} connected non-contractible manifold $M$ for which the persistent homology of $\vrc{M}{-}$ is known for all homological dimensions and at all scales $r$. In \cite{AA-VRS1} it is shown that if $S^1$ is the circle, then $\vrcless{S^1}{r}$ obtains the homotopy types $S^1$, $S^3$, $S^5$, $S^7$, \ldots as $r$ increases, until finally it is contractible. This example agrees with a conjecture of Hausmann \cite[(3.12)]{Hausmann1995} that for $M$ a compact Riemannian manifold, the connectivity of $\vrcless{M}{r}$ is a non-decreasing function of $r$.

We study Vietoris--Rips complexes of ellipses equipped with the Euclidean metric. We restrict attention to ellipses $Y$ of small eccentricity, meaning $Y=\{(x,y)\in \R^2~|~(x/a)^2+y^2=1\}$ with semi-major axis length $a$ satisfying $1<a\le\sqrt{2}$. The threshhold $a\le\sqrt{2}$ guarantees that the intersection of $Y$ with any Euclidean ball centered at a point in $Y$ is connected (Lemma~\ref{lem:d_p-extrema}). Our main result is the following.

\vspace{2mm}
\noindent {\bf Main Theorem (Theorem~\ref{thm:ellipse-homotopy})} Let $Y$ be an ellipse of small eccentricity, and let $r_1=\frac{4\sqrt{3}a}{a^2+3}$ and $r_2=\frac{4\sqrt{3}a^2}{3a^2+1}$. Then
\[
\vrcless{Y}{r}\simeq \begin{cases}
S^1&\mbox{for }0<r\le r_1\\
S^2&\mbox{for }r_1<r\le r_2
\end{cases}
\ \mbox{and}\ 
\vrcleq{Y}{r}\simeq \begin{cases}
S^1&\mbox{for }0<r<r_1\\
S^2&\mbox{for }r=r_1\\
\bigvee^5 S^2&\mbox{for }r_1<r<r_2\\
\bigvee^3 S^2&\mbox{for }r=r_2.
\end{cases}
\]
Furthermore,
\begin{itemize}
\item For $0<r<\tilde{r}\le r_1$ or $r_1<r<\tilde{r}\le r_2$, inclusion $\vrcless{Y}{r}\hookrightarrow\vrcless{Y}{\tilde{r}}$ is a homotopy equivalence.
\item For $0<r<\tilde{r}< r_1$, inclusion  $\vrcleq{Y}{r}\hookrightarrow\vrcleq{Y}{\tilde{r}}$ is a homotopy equivalence.
\item For $r_1\le r<\tilde{r}\le r_2$, inclusion $\vrcleq{Y}{r}\hookrightarrow\vrcleq{Y}{\tilde{r}}$ induces a rank 1 map on 2-dimensional homology $H_2(-;\F)$ for any field $\F$.
\end{itemize}
\vspace{2mm}

\begin{figure}[h]
\begin{center}
\includegraphics[scale=0.8]{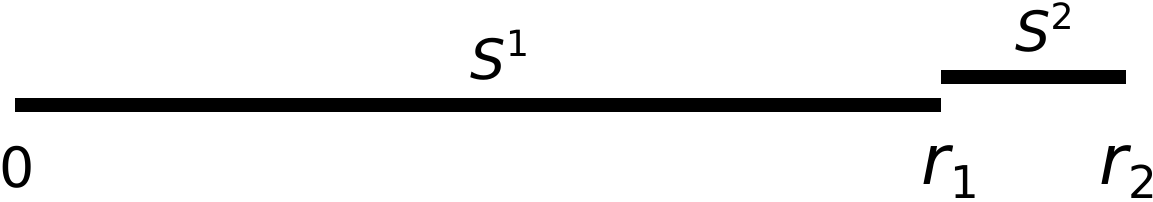} 
\vskip5mm
\includegraphics[scale=0.8]{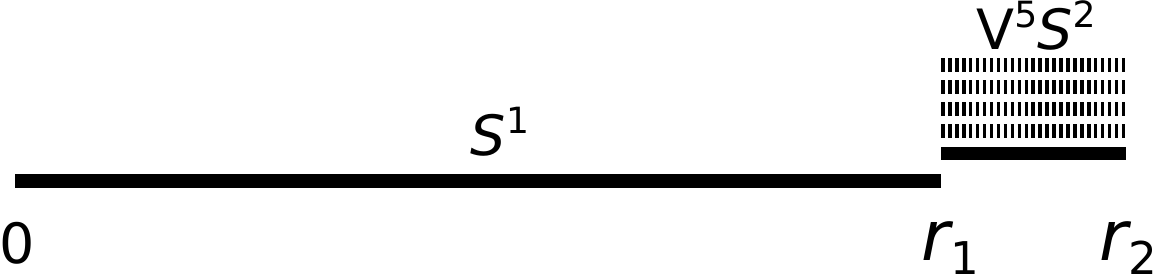}
\end{center}
\caption{The homotopy types of $\vrcless{Y}{r}$ (top) and $\vrcleq{Y}{r}$ (bottom), and also the persistent homology in homological dimensions 1 and 2. See Theorem~\ref{thm:ellipse-homotopy} and Corollary~\ref{cor:ellipse-persistent-homology}.}
\label{fig:PH}
\end{figure}

As a corollary (Corollary~\ref{cor:ellipse-persistent-homology}) we obtain the persistent homology of $\vrc{Y}{r}$ over the interval $[0,r_2)$; see Figure~\ref{fig:PH}. Even though $\dim(H_2(\vrcleq{Y}{r}))=5$ for all $r_1<r<r_2$, note that for all $r_1<r<\tilde{r}<r_2$ we have 
\begin{align*} \rank(H_2(\vrcleq{Y}{r};\F)&\to H_2(\vrcleq{Y}{\tilde{r}};\F))=1\quad\mbox{which implies}\\
\nullity(H_2(\vrcleq{Y}{r};\F)&\to H_2(\vrcleq{Y}{\tilde{r}};\F))=4.
\end{align*}
So the persistent homology of $\vrcleq{Y}{r}$ has an \emph{ephemeral summand} (see \cite{ChazalCrawley-BoeveyDeSilva}) of dimension four over $(r_1,r_2)$.

We believe the main theorem and its corollary are important for the following reason. The theoretical guarantees in (i)-(iii) on recovering information about $M$ from a finite sampling $X$ rely on scale $r$ being ``sufficiently small" depending on $M$, but since $M$ is unknown so is this bound on scale parameters. The parameter-independent stability guarantee in (iv) says that if a sequence of finite samplings $\{X_n\}$ converges $X_n\to M$, then the persistent homology of $\vrc{X_n}{-}$ converges to that of $\vrc{M}{-}$, but very little is known about the behavior of $\vrc{M}{-}$ at larger scales. Suppose for example, that from the Vietoris--Rips complex on a finite dataset $X$, one obtains persistent homology intervals like the one shown in Figure~\ref{fig:sphereOrEllipse}(top), with a single long 1-dimensional interval, and then later a single long 2-dimensional interval. This could be the persistent homology of a dataset $X$ which is sampled from a closed curve wrapping around a sphere in $\R^3$, for example as in Figure~\ref{fig:sphereOrEllipse}(bottom left). Note that as $r$ increases, $\vrc{X}{r}$ will first be homotopy equivalent to a circle $S^1$, and then later to a 2-sphere $S^2$. However, our main theorem shows that another possibility is that dataset $X$ is sampled from an ellipse in $\R^2$, as shown in Figure~\ref{fig:sphereOrEllipse}(bottom right). This is a cautionary tale for data analysts: \emph{persistent} higher-dimensional homology features appearing from Vietoris--Rips complexes are not necessarily evidence that one's dataset is of similarly high dimension.

\begin{figure}[h]
\begin{center}
\includegraphics[scale=0.8]{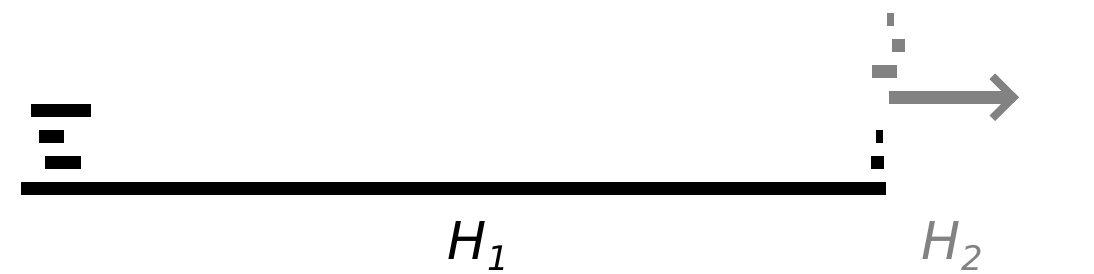} 
\vskip5mm
\includegraphics[scale=0.4]{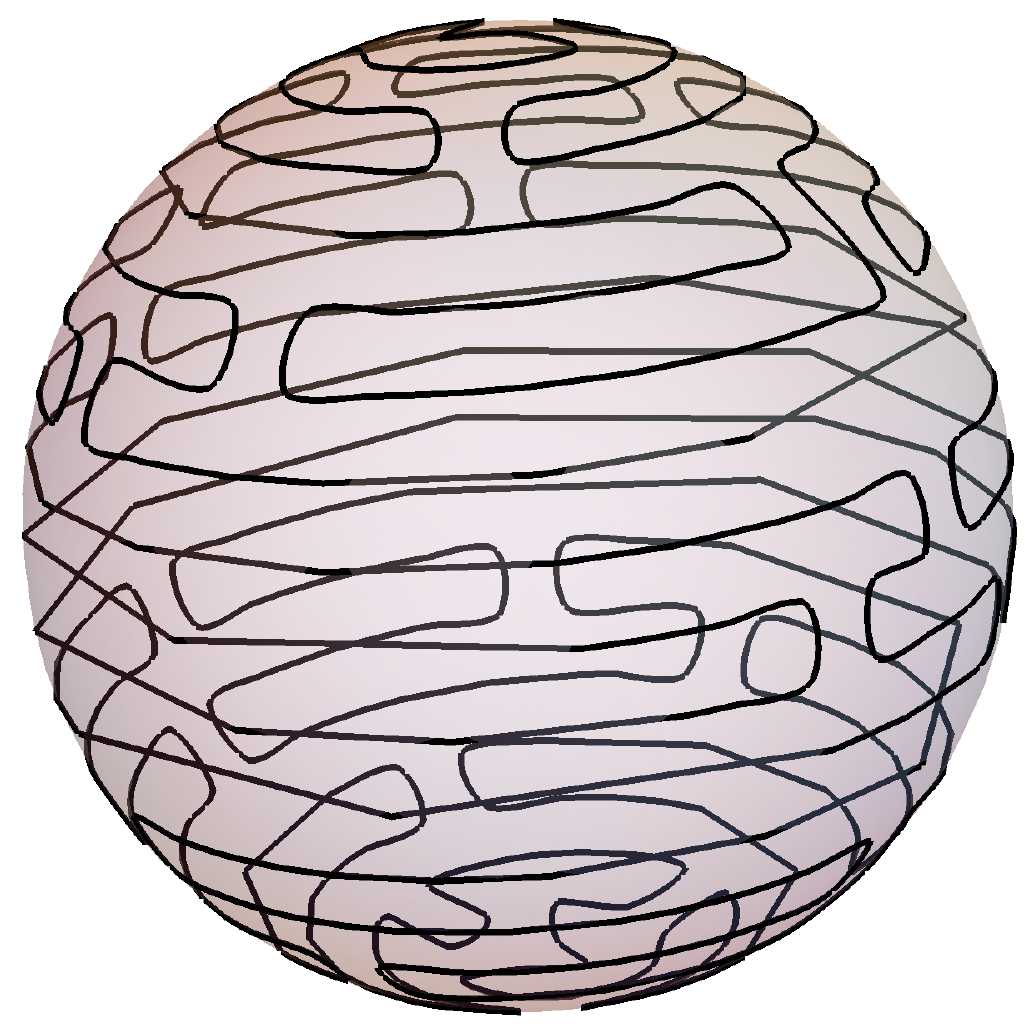}
\hskip5mm
\includegraphics[scale=0.18]{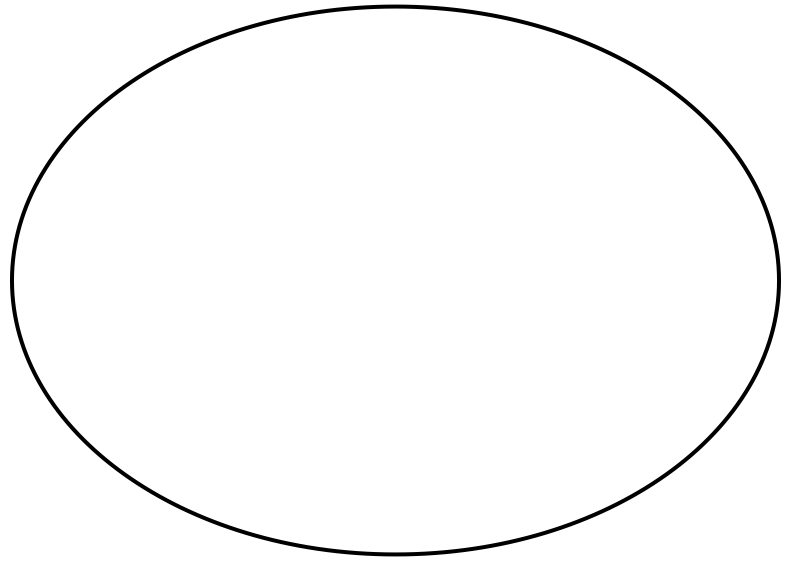}
\end{center}
\caption{The persistent homology of $\vrc{X}{r}$ (top), with 1-dimensional homology in black and 2-dimensional homology in gray, where $X$ is a finite dataset sampled from an underlying space $M$ that could have been either a closed curve on the sphere (bottom left) or from an ellipse (bottom right).}
\label{fig:sphereOrEllipse}
\end{figure}

We also study the behavior of finite samplings from an ellipse; these results appear in the third author's bachelor's thesis \cite{Reddy}. In the case of the circle, the Vietoris--Rips complex of any sufficiently dense sample of the circle recovers the homotopy type of $\vrcless{S^1}{r}$ \cite[Section~5]{AA-VRS1}. This shows that in the particular case when the Riemannian manifold is the circle, an analogue of Latschev's Theorem as stated in (ii) is true not only for small scale parameters $r$ but also for arbitrary scale parameters. However, this fails wildly for the ellipse (which is not a Riemannian manifold when equipped with the Euclidean metric).

\vspace{2mm}
\noindent {\bf Secondary Theorem (Theorem~\ref{thm:exactly-m-orbits})} Let $Y$ be an ellipse of small eccentricity. For any $r_1 < r < r_2$, $\epsilon > 0$, and $m \geq 1$, there exists an $\epsilon$-dense finite subset $X \subseteq Y$ so that $\vrc{X}{r}\simeq\bigvee^m S^2$.
\vspace{2mm}

\noindent Hence even for an arbitrarily dense subset $X$ of the ellipse $Y$, the homotopy types of $\vrc{X}{r}$ and $\vrc{Y}{r}$ need not agree.

In Section~\ref{sec:preliminaries} we provide preliminary concepts and notation on graphs, Vietoris--Rips complexes, and topological spaces. We review finite cyclic graphs and dynamical systems from \cite{AA-VRS1,AAM} in Sections~\ref{sec:finite} and \ref{sec:finite-dynamics}. In Section~\ref{sec:infinite} we generalize this theory to infinite cyclic graphs.  Necessary geometric lemmas about ellipses appear in Section~\ref{sec:geometric-lemmas}. Finally, in Section~\ref{sec:VR} we use the theory of infinite cyclic graphs to prove our results about the homotopy types and persistent homology of Vietoris--Rips complexes of small eccentricity ellipses.

\section{Preliminaries}\label{sec:preliminaries}

We refer the reader to Hatcher \cite{Hatcher} and Kozlov \cite{Kozlov} for basic concepts in topology and combinatorial topology.

\subsection*{Directed graphs}
A \emph{directed graph} is a pair $G=(V,E)$ with $V$ the set of vertices and $E\subseteq V\times V$ the set of directed edges, where we require that there are no self-loops and that no edges are oriented in both directions. The edge $(v,w)$ will also be denoted by $v\diredge w$. A \emph{homomorphism of directed graphs} $h\colon G\to \widetilde{G}$ is a vertex map such that for every edge $v\to w$ in $G$, either $h(v)=h(w)$ or there is an edge $h(v)\to h(w)$ in $\widetilde{G}$. For a vertex $v\in V$ we define the out- and in-neighborhoods
\[ N^+(G,v)=\{w~|~v\diredge w\}, \quad N^-(G,v)=\{w~|~w\diredge v\}, \]
as well as their respective closed versions
\[ N^+[G,v]=N^+(G,v)\cup\{v\}, \quad N^-[G,v]=N^-(G,v)\cup\{v\}. \]

For $W\subseteq V$ we let $G[W]$ denote the induced subgraph on the vertex set $W$. An undirected graph is a graph in which the orientations on the edges are omitted. For an undirected graph $G$ the \emph{clique complex} $\cl(G)$ is the simplicial complex with vertex set $V(G)$ and with faces determined by all cliques (complete subgraphs) of $G$. The clique complex of a directed graph is the clique complex of the underlying undirected graph obtained by forgetting orientations.

\subsection*{Topological spaces} If $X$ and $Z$ are homotopy equivalent topological spaces, then we write $X\simeq Z$. The $i$-fold wedge sum of a space $X$ with itself is denoted $\bigvee^i X$; if $X$ is a nice space such as a connected complex then the choice of basepoint doesn't affect the homotopy type of the wedge sum. The closure of a set $Z\subseteq X$ is denoted $\overline{Z}$. We do not distinguish between an abstract simplicial complex $X$ and its geometric realization, and we let $\redhom_i(X;G)$ denote its $i$-dimensional reduced homology with coefficients in the abelian group $G$.

\subsection*{Vietoris--Rips complexes}
The Vietoris--Rips complex is used to measure the shape of a metric space at a specific choice of scale \cite{EdelsbrunnerHarer,ChazalDeSilvaOudot2013}. 
\begin{definition}
For $X$ a metric space and $r>0$, the \emph{Vietoris--Rips} complex $\vrcless{X}{r}$ (resp.\ $\vrcleq{X}{r}$) is the simplicial complex with vertex set $X$, where a finite subset $\sigma\subseteq X$ is a face if and only if the diameter of $\sigma$ is less than $r$ (resp.\ at most $r$).
\end{definition}

A Vietoris--Rips complex is the \emph{clique} or \emph{flag} complex of its $1$-skeleton. We will omit the subscript and write $\vrc{X}{r}$ in statements which are true for a consistent choice of either subscript $<$ or $\leq$. We denote the $1$-skeleton of a Vietoris--Rips complex by $\vr{X}{r}$.

\subsection*{Conventions for the circle}
Let $S^1$ be the circle of unit circumference equipped with the arc-length metric. If $p_1,p_2,\ldots,p_s\in S^1$ then we write
\[ p_1\preceq p_2\preceq\cdots\preceq p_s\preceq p_1 \]
to denote that $p_1,\ldots,p_s$ appear on $S^1$ in this clockwise order, allowing equality. We may replace $p_i\preceq p_{i+1}$ with $p_i\prec p_{i+1}$ if in addition $p_i\neq p_{i+1}$. For $p,q\in S^1$ we denote the clockwise distance from $p$ to $q$ by $\ddist(p,q)$, and for $p\neq q$ we denote the closed clockwise arc from $p$ to $q$ by $[p,q]_{S^1}=\{z\in S^1~|~p \preceq z \preceq q \preceq p\}$. Open and half-open arcs are defined similarly and denoted $(p,q)_{S_1}$, $(p,q]_{S_1}$, or $[p,q)_{S_1}$.

For a fixed choice of $0\in S^1$ each point $x\in S^1$ can be identified with the real number $\ddist(0,x)\in[0,1)$, and this will be our coordinate system on $S^1$.

\section{Cyclic graphs}\label{sec:finite}

In this section we recall finite cyclic graphs as studied in \cite{Adamaszek2013,AA-VRS1,AAFPP-J}, and we also extend the definitions to include possibly infinite cyclic graphs. Our motivation is that if $Y$ is an ellipse of small eccentricity and $X\subseteq Y$, then $\vr{X}{r}$ is a cyclic graph.

The following definition generalizes \cite[Definition~3.1]{AA-VRS1} to the possibly infinite case.

\begin{definition}
A directed graph with vertex set $V\subseteq S^1$ is \emph{cyclic} if whenever there is a directed edge $v\diredge u$, then there are directed edges $v\diredge w$ and $w\diredge u$ for all $v\prec w\prec u\prec v$.
\end{definition}

\begin{figure*}[h]
    \begin{center}
 \begin{tikzpicture}[
     decoration={
       markings,
       mark=at position 1 with {\arrow[scale=2,black]{latex}};
     },every node/.style={draw,circle}
   ]
\def \n {5}
\def \radius {3cm}
\def \margin {8} 
\foreach \s in {0,...,5}
{\draw (-60*\s: 1.7cm) node(\s){\s};}
    \draw [postaction=decorate] (0) -- (1);
    \draw [postaction=decorate] (0) -- (2);
    \draw [postaction=decorate] (1) -- (2);
    \draw [postaction=decorate] (1) -- (3);
    \draw [postaction=decorate] (1) -- (4);
    \draw [postaction=decorate] (2) -- (3);
    \draw [postaction=decorate] (2) -- (4);
    \draw [postaction=decorate] (3) -- (4);
    \draw [postaction=decorate] (4) -- (5);
    \draw [postaction=decorate] (5) -- (0);
\end{tikzpicture}
\hspace{1mm}
 \begin{tikzpicture}[
     decoration={
       markings,
       mark=at position 1 with {\arrow[scale=2,black]{latex}};
     },every node/.style={draw,circle}
   ]
\def \n {5}
\def \radius {3cm}
\def \margin {8} 
\foreach \s in {0,...,7}
{\draw (-45*\s: 1.7cm) node(\s){\s};}
    \draw [postaction=decorate] (0) -- (1);
    \draw [postaction=decorate] (0) -- (2);
    \draw [postaction=decorate] (1) -- (2);
    \draw [postaction=decorate] (1) -- (3);
    \draw [postaction=decorate] (1) -- (4);
    \draw [postaction=decorate] (2) -- (3);
    \draw [postaction=decorate] (2) -- (4);
    \draw [postaction=decorate] (2) -- (5);
    \draw [postaction=decorate] (3) -- (4);
    \draw [postaction=decorate] (3) -- (5);
    \draw [postaction=decorate] (3) -- (6);
    \draw [postaction=decorate] (4) -- (5);
    \draw [postaction=decorate] (4) -- (6);
    \draw [postaction=decorate] (5) -- (6);
    \draw [postaction=decorate] (5) -- (7);
    \draw [postaction=decorate] (5) -- (0);
    \draw [postaction=decorate] (6) -- (7);
    \draw [postaction=decorate] (6) -- (0);
    \draw [postaction=decorate] (6) -- (1);
    \draw [postaction=decorate] (7) -- (0);
    \draw [postaction=decorate] (7) -- (1);
\end{tikzpicture}
\hspace{1mm}
\begin{tikzpicture}[
     decoration={
       markings,
       mark=at position 1 with {\arrow[scale=1.5,black]{latex}};
     },every node/.style={draw,circle}
   ]
\def \n {5}
\def \radius {3cm}
\def \margin {8} 
\foreach \s in {0,...,8}
{\draw (-40*\s: 1.7cm) node(\s){\s};}
    \draw [postaction=decorate] (0) -- (1);
    \draw [postaction=decorate] (1) -- (2);
    \draw [postaction=decorate] (2) -- (3);
    \draw [postaction=decorate] (3) -- (4);
    \draw [postaction=decorate] (4) -- (5);
    \draw [postaction=decorate] (5) -- (6);
    \draw [postaction=decorate] (6) -- (7);
    \draw [postaction=decorate] (7) -- (8);
    \draw [postaction=decorate] (8) -- (0);
    \draw [postaction=decorate] (0) -- (2);
    \draw [postaction=decorate] (1) -- (3);
    \draw [postaction=decorate] (2) -- (4);
    \draw [postaction=decorate] (3) -- (5);
    \draw [postaction=decorate] (4) -- (6);
    \draw [postaction=decorate] (5) -- (7);
    \draw [postaction=decorate] (6) -- (8);
    \draw [postaction=decorate] (7) -- (0);
    \draw [postaction=decorate] (8) -- (1);
    \draw [postaction=decorate] (0) -- (3);
    \draw [postaction=decorate] (1) -- (4);
    \draw [postaction=decorate] (2) -- (5);
    \draw [postaction=decorate] (3) -- (6);
    \draw [postaction=decorate] (4) -- (7);
    \draw [postaction=decorate] (5) -- (8);
    \draw [postaction=decorate] (6) -- (0);
    \draw [postaction=decorate] (7) -- (1);
    \draw [postaction=decorate] (8) -- (2);
\end{tikzpicture}
    \end{center}
    \caption{(Left) A cyclic graph of winding fraction $\tfrac{1}{4}$; dominated vertex 1 is fast (Definition~\ref{def:fast-slow}) and 3 is slow. (Middle) A cyclic graph of winding fraction $\frac{1}{3}$; dominated vertex 3 is fast and 7 is slow. (Right) The regular cyclic graph $\cnk{9}{3}$.}
    \label{fig:cyclic}
\end{figure*}
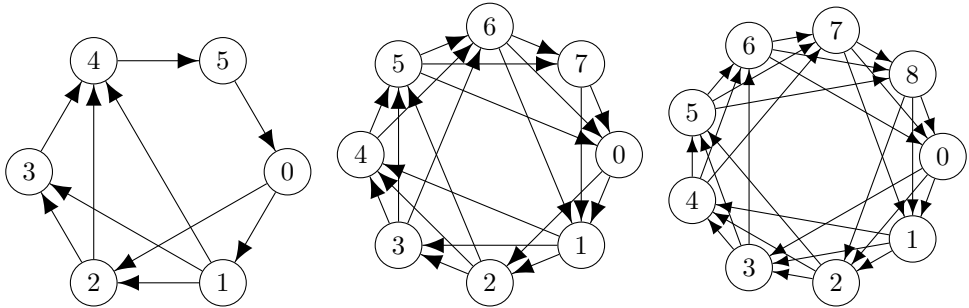

The simplest examples of cyclic graphs are the finite regular cyclic graphs.

\begin{definition}
For integers $n$ and $k$ with $0 \leq k < \frac{1}{2} n$, the cyclic graph $\cnk{n}{k}$ has vertex set $\{0,\ldots,n-1\}$ and edges $i \diredge (i+s) \mod n$ for all $i = 0, \ldots, n-1$ and $s = 1, \ldots, k$.
\end{definition}

We also generalize cyclic graphs homomorphisms (Definition~3.5 of \cite{AA-VRS1}) to the possibly infinite case.

\begin{definition}
Let $G$ and $\widetilde{G}$ be cyclic graphs. A homomorphism $h:G\to \widetilde{G}$ of directed graphs is \emph{cyclic} if
\begin{itemize}
\item whenever $v\prec w\prec u\prec v$ in $G$, then $h(v)\preceq h(w)\preceq h(u)\preceq h(v)$ in $\widetilde{G}$, and
\item $h$ is not constant whenever $G$ has a directed cycle of edges.
\end{itemize}
\end{definition}

An important numerical invariant of a cyclic graph is the \emph{winding fraction}, given for the finite case in \cite[Definition~3.7]{AA-VRS1}.

\begin{definition}\label{def:windingfraction}
The \emph{winding fraction} of a finite cyclic graph $G$ is
\[ \wf(G) = \sup\Big\{\frac{k}{n}~:~\mbox{there exists a cyclic homomorphism }\cnk{n}{k}\to G \Big\}. \]
\end{definition}

If $G$ is a cyclic graph on vertex set $V\subseteq S^1$, and if $W\subseteq V$ is a vertex subset, then note $G[W]$ is also a cyclic graph.

\begin{definition}\label{def:windingfraction-inf}
The \emph{winding fraction} of a (possibly infinite) cyclic graph $G$ is
\[ \wf(G)=\sup_{W\subseteq V\mbox{ finite}} \wf(G[W]). \]
\end{definition}
\noindent We say the \emph{supremum is attained} if there exists some finite $W$ with $\wf(G) = \wf(G[W])$. Note that if the supremum is attained, then the winding fraction is necessarily rational.

If cyclic graph $G$ is finite, then we often label the vertices in cyclic order $v_0\prec v_1\prec\ldots\prec v_{n-1}$.

\begin{definition}
\label{def:dominated}
A vertex $v_i$ in a finite cyclic graph $G$ is \emph{dominated (by $v_{i+1}$)} if $N^-(G,v_{i+1})=N^-[G,v_i]$.
\end{definition}

We may \emph{dismantle} a finite cyclic graph by removing dominated vertices in rounds as follows. Let $S_0$ be the set of all dominated vertices in $G$, and in round $0$ remove these vertices to obtain $G[V\setminus S_0]$. Inductively for $i>0$, let $S_i$ be the set of all dominated vertices in $G[V\setminus(S_0 \cup \cdots \cup S_{i-1})]$, and in round $i$ remove these vertices to obtain $G[V\setminus(S_0 \cup \cdots \cup S_i)]$. Removing a dominated vertex does not change the homotopy type of the clique complex, and hence if $G$ dismantles to $G[W]$ for $W\subseteq V$ then we have a homotopy equivalence $\cl(G)\simeq\cl(G[W])$. By \cite[Proposition~3.12]{AA-VRS1} every finite cyclic graph dismantles down to an induced subgraph of the form $C_n^k$ for some $0 \le k < \frac{1}{2}n$. It follows from \cite[Proposition~3.14]{AA-VRS1} that the winding fraction of a finite cyclic graph $G$ can be equivalently defined as $\wf(G)=\frac{k}{n}$, where $G$ dismantles to $\cnk{n}{k}$.

\section{Finite cyclic dynamical systems}\label{sec:finite-dynamics}

Associated to each finite cyclic graph is a finite cyclic dynamical system, as studied in \cite{AAM}.

\begin{definition}
Let $G$ be a finite cyclic graph with vertex set $V$. The associated \emph{finite cyclic dynamical system} is generated by the dynamics $f \colon V \to V$, where $f(v)$ is defined to be the clockwise most vertex of $N^+[G,v]$.
\end{definition}

A vertex $v\in V$ is \emph{periodic} if $f^i(v)=v$ for some $i\ge1$. If $v$ is periodic then we refer to $\{f^i(v)~|~i\ge0\}$ as a \emph{periodic orbit}, whose \emph{length} $\olen$ is the smallest $i\ge1$ such that $f^i(v)=v$. The \emph{winding number} of a periodic orbit is $\sum_{i=0}^{\olen-1}\ddist(f^i(v),f^{i+1}(v))$. It follows from \cite[Lemma~2.3]{AAM} that in a finite cyclic graph, every periodic orbit has the same length $\olen$ and the same winding number $\owind$. We let $\numorb$ denote the number of periodic orbits in $G$.

By \cite[Proposition~3.12]{AA-VRS1} any finite cyclic graph dismantles down to an induced subgraph of the form $\cnk{n}{k}$. Since the proof of \cite[Lemma~6.3]{AAM} holds in the case of arbitrary finite cyclic graphs, rather than solely for Vietoris--Rips graphs of the circle, we see that the vertex set of this induced subgraph is the set of periodic vertices of $f$. Thus for a finite cyclic graph we have $\wf(G)=\frac{k}{n}=\frac{\owind}{\olen}$, so the winding fraction of $G$ is determined by the dynamical system $f$. We can describe the homotopy type of the clique complex of $G$ via the dynamical system $f$ as follows, where $\N$ denotes the set of nonnegative integers. 

\begin{proposition}\label{prop:vrchtpy}
If $G$ is a finite cyclic graph, then
\[ \cl(G)\simeq\begin{cases}
S^{2l+1} & \mbox{if }\frac{l}{2l+1}<\wf(G)<\frac{l+1}{2l+3}\mbox{ for some }l\in\N,\\
\bigvee^{\numorb-1}S^{2l} & \mbox{if }\wf(G)=\frac{l}{2l+1}.
\end{cases} \]
\end{proposition}

\begin{proof}
Theorem~4.4 of \cite{AA-VRS1} states 
\[ \cl(G)\simeq\begin{cases}
S^{2l+1} & \mbox{if }\frac{l}{2l+1}<\wf(G)<\frac{l+1}{2l+3}\mbox{ for some }l\in\N,\\
\bigvee^{n-2k-1}S^{2l} & \mbox{if }\wf(G)=\frac{l}{2l+1}\mbox{ and }G\mbox{ dismantles to }\cnk{n}{k}.
\end{cases} \]
Suppose $G$ dismantles to $\cnk{n}{k}$ and
\begin{equation}\label{eq:wf-equality}
\frac{l}{2l+1}=\wf(G)=\frac{k}{n}.
\end{equation}
Then by \eqref{eq:wf-equality} we have $\numorb=\gcd(n,k)=\frac{k}{l}=n-2k$.
\end{proof}

\begin{definition}
Let $G$ be a finite cyclic graph with dynamics $f\colon V\to V$, let $\widetilde{G}$ be a finite cyclic graph with dynamics $\tilde{f}\colon \widetilde{V}\to \widetilde{V}$, let $\wf(G)=\frac{\omega}{\ell}=\wf(\widetilde{G})$, and let $h \colon G \to \widetilde{G}$ be a cyclic graph homomorphism. We say that a periodic orbit of $\widetilde{G}$ with vertex set $\{\tilde{v}_1, \ldots, \tilde{v}_\ell\}$ is \emph{hit by $h$} if there is a periodic orbit of $G$ with vertex set $\{v_1, \ldots, v_\ell\}$ such that $\tilde{f}^i(h(\{v_1, \ldots, v_\ell\}))=\{\tilde{v}_1, \ldots, \tilde{v}_\ell\}$ for some $i$.
\end{definition}

It is clear that each periodic orbit of $G$ hits a periodic orbit of $\widetilde{G}$. Let $\hit(h)$ be the number of distinct periodic orbits of $\widetilde{G}$ hit by $h$.

\begin{proposition}\label{prop:rank}
Let $h\colon G \to \widetilde{G}$ be a cyclic graph homomorphism of finite cyclic graphs with $\wf(G)=\frac{l}{2l+1}=\wf(\widetilde{G})$. Then there are bases $\{b_1,\ldots,b_{d-1}\}$ for $\redhom_{2l}(\cl(G);\Z)=\Z^{d-1}$ and $\{\tilde{b}_1,\ldots,\tilde{b}_{\tilde{d}-1}\}$ for $\redhom_{2l}(\cl(\widetilde{G});\Z)=\Z^{\tilde{d}-1}$ such that the map on homology $h^*\colon \redhom_{2l}(\cl(G);\Z) \to \redhom_{2l}(\cl(\widetilde{G});\Z)$ induced by $h$ satisfies 
\begin{itemize}
\item for all $1\le i\le d-1$, either $h_*(b_i)=0$ or $h_*(b_i)=\tilde{b_j}$ for some $j$, and
\item $|\{j~|~h_*(b_i)=\tilde{b_j}\mbox{ for some }i\}|=\hit(h)-1$. 
\end{itemize}
If $d=\hit(h)=\tilde{d}$, then $h$ induces a homotopy equivalence of clique complexes. Furthermore, for any field $\F$, the map $\redhom_{2l}(\cl(G);\F) \to \redhom_{2l}(\cl(\widetilde{G});\F)$ induced by $h$ has rank $\hit(h)-1$.
\end{proposition}

\begin{proof}
Since $\wf(G)=\frac{l}{2l+1}$ it follows that $G$ dismantles to $\cnk{d(2l+1)}{dl}$ for some $d\ge1$. Label the vertices of $\cnk{d(2l+1)}{dl}$ as $v_0 \prec v_1 \prec \ldots \prec v_{d(2l+1)-1}\prec v_0$; each periodic orbit of $G$ is of the form $\sigma_i=\{v_i,v_{i+d},\ldots,v_{i+2ld}\}$ for $0\le i\le d-1$. From the proof of \cite[Proposition~8.4]{AA-VRS1} we know that one basis for $\redhom_{2l}(\cl(G);\Z)=\Z^{d-1}$ is given by the cross-polytopal elements $\{e_1, \ldots, e_{d-1}\}$, where the vertex set of $e_i$ is $\sigma_0\cup\sigma_i$.

Let $\tilde{f}$ be the dynamical system corresponding to $\widetilde{G}$. Choose $N$ sufficiently large so that the image of $\tilde{f}^N\colon \widetilde{V} \to \widetilde{V}$ is contained in the set of periodic vertices of $\tilde{f}$. Relabel the vertices of $\tilde{f}^N(\widetilde{V})$ as $\tilde{v}_0 \prec \tilde{v}_1 \prec \ldots \prec \tilde{v}_{\tilde{d}(2l+1)-1}\prec \tilde{v}_0$ so that $\tilde{f}^N(\sigma_0)=\tilde{\sigma}_0$, where $\tilde{\sigma}_i=\{\tilde{v}_i,\tilde{v}_{i+\tilde{d}},\ldots,\tilde{v}_{i+2l\tilde{d}}\}$. A basis for $\redhom_{2l}(\cl(\widetilde{G});\Z)=\Z^{\tilde{d}-1}$ is given by the cross-polytopal elements $\{\tilde{e}_1, \ldots, \tilde{e}_{\tilde{d}-1}\}$, where the vertex set of $\tilde{e}_i$ is $\tilde{\sigma}_0\cup\tilde{\sigma}_i$; this follows since $\widetilde{G}$ dismantles down to $\cnk{\tilde{d}(2l+1)}{\tilde{d}l}$ with vertices labeled as above. Furthermore, $\tilde{f}^N\colon \widetilde{V} \to \widetilde{V}$ induces a homotopy equivalence $\cl(\widetilde{G})\xrightarrow{\simeq}\cl(\widetilde{G}[\tilde{f}^N(\widetilde{V})])$. Indeed, $\tilde{f}^n$ induces an isomorphism on homology by the choice of basis above, and so $\tilde{f}^n$ induces a homotopy equivalence by the theorems of Hurewicz and Whitehead.

For $1\le i,j\le d-1$ we have
\begin{align*}
\tilde{f}^N h(\sigma_i)=\tilde{\sigma}_j&\quad\mbox{implies}\quad(\tilde{f}^N h)^*(e_i)=\tilde{e}_j,\quad\mbox{and}\\
\tilde{f}^N h(\sigma_i)=\tilde{\sigma}_0&\quad\mbox{implies}\quad(\tilde{f}^N h)^*(e_i)=0.
\end{align*}
It follows from the definition of $\hit(h)$ that $|\{j~|~h_*(b_i)=\tilde{b_j}\mbox{ for some }i\}|=\hit(h)-1$, giving the desired bases.

If $d=\hit(h)=\tilde{d}$, then $h^*$ is an isomorphism on homology, and so $h$ is a homotopy equivalence by the theorems of Hurewicz and Whitehead. The claim about ranks over a field $\F$ follows from the statement for integer coefficents.
\end{proof}

\section{Infinite cyclic graphs}\label{sec:infinite}

We now study the homotopy types of all cyclic graphs, finite or infinite. We say a cyclic graph $G$ is \emph{generic} if $\frac{l}{2l+1}<\wf(G)<\frac{l+1}{2l+3}$ for some $l\in\N$, or if $\wf(G)=\frac{l+1}{2l+3}$ and the supremum in Definition~\ref{def:windingfraction-inf} is not attained. Otherwise, we say that $G$ is \emph{singular}.

For an arbitrary undirected graph $G$, the clique complex $\cl(G)$ is given the topology of a CW-complex. More precisely, this is the weak topology with respect to finite-dimensional skeleta, or equivalently, with respect to subcomplexes induced by finite subsets of $V(G)$. Formally, let $\fin(G)$ be the poset of all finite subsets of $V(G)$ ordered by inclusion. We have a functor $\cl(-)\colon \fin(G) \to Top$ and
\[ \cl(G)=\colim_{W \in \fin(G)} \cl(G[W])\simeq \hocolim_{W \in \fin(G)}\cl(G[W]). \]
The last homotopy equivalence follows from \cite[Section~3]{Hocolim} and the fact that all maps $\cl(G[W])\hookrightarrow\cl(G[\widetilde{W}])$ for $W\subseteq \widetilde{W}$ are inclusions of closed subcomplexes, hence cofibrations. For a finite subset $W_0\subseteq V(G)$ let $\fin(G;W_0)$ be the subposet of $\fin(G)$ consisting of all sets which contain $W_0$. This poset is cofinal in $\fin(G)$, and hence
\[ \cl(G)=\colim_{W \in \fin(G;W_0)} \cl(G[W])\simeq \hocolim_{W \in \fin(G;W_0)}\cl(G[W]). \]

\begin{theorem}\label{thm:homotopy-generic}
Let $G$ be a generic cyclic graph with $\frac{l}{2l+1}<\wf(G)\le\frac{l+1}{2l+3}$ for $l\in\N$. Then $\cl(G)\simeq S^{2l+1}$.

Moreover, if $h\colon G\hookrightarrow\widetilde{G}$ is an inclusion of generic cyclic graphs with identical vertex sets $V(G)=V(\widetilde{G})$ and with $\frac{l}{2l+1}<\wf(G)\le\wf(\widetilde{G})\le\frac{l+1}{2l+3}$, then $h$ induces a homotopy equivalence of clique complexes.
\end{theorem}

\begin{proof}
Our proof uses the same tools as \cite[Proposition~7.1]{AA-VRS1}. The hypotheses on $G$ imply there exists a finite subset $W_0 \subseteq V$ such that for every finite subset $W$ with $W_0 \subseteq W \subseteq V$, we have $\frac{l}{2l+1}<\wf(G[W])<\frac{l+1}{2l+3}$. By \cite[Proposition~4.9]{AA-VRS1} all maps in the diagram $\cl(G[-])\colon \fin(G;W_0)\to Top$ are homotopy equivalences between spaces homotopy equivalent to $S^{2l+1}$. Hence
\[ \cl(G)\simeq \hocolim_{W \in \fin(G;W_0)}\cl(G[W])\simeq S^{2l+1} \]
by the Quasifibration Lemma \cite[Proposition~3.6]{Hocolim}.

The same is true for $\widetilde{G}$. Furthermore, the maps $\cl(G[W])\to\cl(\widetilde{G}[W])$ induced by $h$ now define a natural transformation of diagrams $\cl(G[-])\to\cl(\widetilde{G}[-])$ which is a levelwise homotopy equivalence by \cite[Proposition~4.9]{AA-VRS1}. It follows that the induced map of (homotopy) colimits is a homotopy equivalence.
\end{proof}

\begin{theorem}\label{thm:homotopy-singular-I}
Let $G$ be a singular cyclic graph with $\wf(G)=\frac{l}{2l+1}$. Then $\cl(G)\simeq\bigvee^\kappa S^{2l}$ for some cardinal number $\kappa$.
\end{theorem}

\begin{proof}
Let $W_0$ be a subset of $V$ such that $\wf(G[W_0])=\frac{l}{2l+1}$. By Proposition~\ref{prop:vrchtpy}, every space in the diagram $\cl(G) = \colim_{W\in \fin(G;W_0)}\cl(G[W])$ is homotopy equivalent to a finite wedge sum of $2l$-spheres. It follows immediately that $\cl(G)$ is simply-connected and its homology is torsion-free and concentrated in degree $2l$. Furthermore, \cite[Proposition~8.4]{AA-VRS1} and the proof technique of \cite[Proposition~7.5]{AA-VRS1} show that the group $\redhom_{2l}(\cl(G);\Z)$ is free abelian. Hence $\cl(G)$ is a model of the Moore space $M(\bigoplus^\kappa \Z,2l)$, unique up to homotopy and equivalent to $\bigvee^\kappa S^{2l}$ for some cardinal number $\kappa$.
\end{proof}

In Theorem~\ref{thm:homotopy-singular-II} we refine this theorem, for certain singular cyclic graphs $G$, to describe the value of $\kappa$ in $\cl(G)\simeq\bigvee^\kappa S^{2l}$. But before we can do so, we first need more properties about infinite cyclic graphs.

\subsection{Infinite cyclic graph properties}

In the case of finite cyclic graphs, the language of cyclic dynamical systems from Section~\ref{sec:finite-dynamics} proved to be convenient. Unfortunately, not all infinite cyclic graphs can be equipped with an analogous dynamical system, since it is possible that the forward neighborhood $N^+[G,v]$ has no clockwise-most vertex, for example if $N^+[G,v]$ is an open subset of $S^1$.

\begin{definition}\label{def:gamma_m}
Let $G$ be a cyclic graph with vertex set $V\subseteq S^1$, and let $m\ge 1$. We define the map $\gamma_m\colon V\to\R$ by
\[ \gamma_m(v_0)=\sup\Biggl\{\sum_{i=0}^{m-1}\ddist(v_i,v_{i+1})~\Bigg|~\mbox{ there exists }v_0\diredge v_1\diredge\ldots\diredge v_m\Biggr\}. \]
The map $f_m\colon V\to S^1$ is defined by $f_m(v)=(v+\gamma_m(v))\md 1$.
\end{definition}

\begin{remark}
If $N^+[g,v]$ is closed for all $v\in V$, then as in the finite case we can define a dynamical system $f \colon V \to V$, where $f(v)$ is defined to be the clockwise-most vertex of $N^+[G,v]$. In this case we have $f_m(v)=f^m(v)\in V$ for all $v\in V$.
\end{remark}

\begin{lemma}\label{lem:wn-discrepancy}
Let $G$ be a cyclic graph, let $v,u\in V(G)$, and let $m\in\N$. Then $|\gamma_m(v)-\gamma_m(u)|<1$.
\end{lemma}

\begin{proof}
Let $v\neq u$. The fact that $G$ is cyclic implies\footnote{Of course the bound $\gamma_m(v)-\gamma_m(u)\le\ddist(u,v)$ may be false.}
\[\gamma_m(v)-\gamma_m(u)\le\ddist(v,u)<1\quad\mbox{and}\quad\gamma_m(u)-\gamma_m(v)\le\ddist(u,v)<1.\]
\end{proof}

\begin{lemma}\label{lem:wf-as-limit}
Let $G$ be a cyclic graph. For any $v\in V$ we have 
\begin{equation}\label{eq:wf}
\wf(G) = \lim_{m\to\infty} \frac{\gamma_m(v)}{m}.
\end{equation}
\end{lemma}

\begin{proof}
As a first case, let $G$ be finite. Recall that $G$ dismantles to a cyclic graph $\cnk{n}{k}$. For any periodic $v\in V$ we have $\lfloor \gamma_m(v)\rfloor=\lfloor\frac{mk}{n}\rfloor$, and hence $\lim_{m\to\infty} \frac{\gamma_m(v)}{m}=\frac{k}{n}=\wf(G)$. Lemma~\ref{lem:wn-discrepancy} then gives that \eqref{eq:wf} also holds for the non-periodic vertices.

We next consider the case when $G$ is infinite. We claim that $\frac{\lfloor \gamma_m(v_0)\rfloor}{m}\le\wf(G)$ for all $m$ and $v_0\in V$. Indeed, suppose for a contradiction that there were some $v_0\diredge\ldots\diredge v_m$ with $\frac{\lfloor \gamma_m(v_0)\rfloor}{m}=:\omega>\wf(G)$. Let $W=\{v_0,\ldots,v_m\}$, and let $\gamma_m^W\colon W\to \R$ denote the map defined in Definition~\ref{def:gamma_m} for cyclic graph $G[W]$. Since $G[W]$ is cyclic, we have $\frac{\lfloor \gamma_{im}^W(v_0)\rfloor}{im}\ge\omega$ for all $i$. Since we have proven \eqref{eq:wf} in the case where $G[W]$ is finite, this gives
\[\wf(G[W]) = \lim_{m'\to\infty} \frac{\gamma_{m'}^W(v_0)}{m'}=\lim_{i\to\infty} \frac{\lfloor \gamma_{im}^W(v_0)\rfloor}{im}\ge\omega>\wf(G), \]
contradicting Definition~\ref{def:windingfraction-inf}. Hence $\frac{\lfloor \gamma_m(v_0)\rfloor}{m}\le\wf(G)$.

Let $\epsilon>0$ be arbitrary. Then there is some finite $W\subseteq V$ with $\wf(G)-\epsilon\le\wf(G[W])=\frac{k}{n}$,
where finite cyclic graph $G[W]$ dismantles down to $\cnk{n}{k}$. Let $u$ be any periodic vertex of $G[W]$; then $\lfloor \gamma_m^W(u)\rfloor=\lfloor\frac{mk}{n}\rfloor$. Hence for any $m$ we have
\[ \wf(G)-\epsilon-\frac{1}{m}\le\frac{k}{n}-\frac{1}{m}\le\frac{\lfloor \frac{mk}{n}\rfloor}{m}=\frac{\lfloor \gamma_m^W(u)\rfloor}{m} \le\frac{\lfloor \gamma_m(u)\rfloor}{m}\le\wf(G). \]
Lemma~\ref{lem:wn-discrepancy} implies that for any $v\in V$ we have $\wf(G)-\epsilon-\frac{2}{m}\le\frac{\lfloor \gamma_m(v)\rfloor}{m}\le\wf(G)$. Since this is true for any $\epsilon>0$, equation \eqref{eq:wf} follows.
\end{proof}

\subsection{Cyclic graphs with rational winding fraction}

In preparation for the setting when $\wf(G)=\frac{l}{2l+1}$, we first study the more general case when $\wf(G)=\frac{p}{q}$ is rational. Whenever we write $\wf(G)=\frac{p}{q}$, we assume that integers $p$ and $q$ are relatively prime.

\begin{definition}
We say $v\in V$ is \emph{periodic} if there is some $i\ge 1$ such that $\gamma_i(v)\in\N$ and the supremum defining $\gamma_i(v)$ is achieved. As in the finite case, the \emph{length} $\olen$ of this periodic orbit is the smallest such $i$, and its \emph{winding number} is $\gamma_\olen(v)$.
\end{definition}

\begin{lemma}\label{lem:periodic}
Let $G$ be cyclic with $\wf(G)=\frac{p}{q}$. If $v\in V$ is periodic, then $v$ has orbit length $q$ and winding number $\gamma_q(v)=p$.
\end{lemma}

\begin{proof}
Let $v$ be periodic with orbit length $q'$ and with $p'=\gamma_{q'}(v)$. Suppose for a contradiction that $d=\gcd(q',p')>1$. Then either $\gamma_{q'/d}(v)\neq \frac{p'}{d}$ or $\gamma_{q'/d}(v)=\frac{p'}{d}$ and the supremum is not obtained, which since $G$ is cyclic implies either $\gamma_{q'}(v)\neq p'$ or $\gamma_{q'}(v)=p'$ and the supremum is not obtained, a contradiction. It follows that $q'$ and $p'$ are relatively prime.

Since $v$ is periodic, by Lemma~\ref{lem:wf-as-limit} we have $\frac{p'}{q'} = \frac{\gamma_{q'}(v)}{q'} = \lim_{m\to\infty} \frac{\gamma_m(v)}{m} = \wf(G) = \frac{p}{q}$.
Since $\gcd(q',p')=1=\gcd(q,p)$, it follows that $q=q'$ and $p=p'$.
\end{proof}

Let $V_\numorb\subseteq V$ be the set of all periodic points.

\begin{definition}\label{def:fast-slow}
We say a non-periodic vertex $v\in V\setminus V_\numorb$ is \emph{fast} if $\gamma_q(v)>p$, and \emph{slow} if $\gamma_q(v)\le p$.
\end{definition}
\noindent See Figure~\ref{fig:cyclic} for example fast and slow vertices. When $G$ is infinite, it is possible for a non-periodic slow vertex to satisfy $\gamma_q(v)=p$.

\begin{definition}\label{def:achieve-periodicity}
A fast vertex $v \in V$ \emph{achieves periodicity} if there is some $m$ such that $f_m(v)$ is in $V$, $f_m(v)$ is periodic, and furthermore the supremum defining $\gamma_m(v)$ is obtained. Otherwise we say $v$ is \emph{permanently fast}.
\end{definition}

We emphasize that by definition, a periodic point can neither be fast nor slow. Hence by our conventions a periodic point does not achieve periodicity; it is already periodic. It is possible for a fast point $v$ to achieve periodicity even if not all $f_j(v)$ are in $V$, just as it is possible for a point $v$ to be periodic even if not all $f_j(V)$ are in $V$.

\subsection{Closed and continuous cyclic graphs}

We now restrict attention to a subset of cyclic graphs.

\begin{definition}\label{def:cont}
We say a cyclic graph $G$ is \emph{closed} if its vertex set $V$ is closed in $S^1$. We say $G$ is \emph{continuous} if $\gamma_m\colon V\to \R$ is continuous for all $m\ge 1$.
\end{definition}

Note that a finite cyclic graph is both closed and continuous.

\begin{lemma}\label{lem:cont-prop}
Let $G$ be a closed and continuous cyclic graph with $\wf(G)=\frac{p}{q}$. Then for all $v\in V$ and $i,j\ge 0$, we have 
\begin{enumerate}
\item[(i)] $f_i(v)\in V$,
\item[(ii)] $\gamma_{i+j}(v)=\gamma_i(v)+\gamma_j(f_i(v))$,
\item[(iii)] $f_{i+j}(v)=f_j(f_i(v))$,
\end{enumerate}
\end{lemma}

\begin{proof}
For (i), we have $f_i(v)\in V$ since $f_i(v)$ is a clockwise supremum of points in the closed set $V$.

For (ii), note that the right-hand side is well-defined by (i). Let $v_0\in V$. The inequality $\gamma_{i+j}(v_0)\le \gamma_i(v_0)+\gamma_j(f_i(v_0))$ is clear. To show the reverse inequality, let $\epsilon>0$. Since $\gamma_j$ is continuous, there exists some $\delta>0$ such that if $v_i\in(f_i(v_0)-\delta,f_i(v_0)]_{S^1}$, then there is some directed path $v_i\diredge\ldots\diredge v_{i+j}$ with $\gamma_j(f_i(v_0))-\sum_{k=i}^{i+j-1}\ddist(v_k,v_{k+1})<\epsilon$. By definition of $f_i(v_0)$ there exists some directed path $v_0\diredge\ldots\diredge v_i$ with $\sum_{k=0}^{i-1}\ddist(v_k,v_{k+1})>\gamma_i(v_0)-\delta$. Concatenating these two paths gives $\sum_{k=0}^{i+j-1}\ddist(v_k,v_{k+1})> \gamma_i(v_0)+\gamma_j(f_i(v_0))-\epsilon - \delta$.
Since this is true for all $\epsilon$ and for any $\delta$ sufficiently small (depending on $\epsilon$), we have $\gamma_{i+j}(v_0)\ge \gamma_i(v_0)+\gamma_j(f_i(v_0))$, giving (ii). Note (iii) follows immediately from (ii). 
\end{proof}

\begin{lemma}
Let $G$ be a closed and continuous cyclic graph with $\wf(G)=\frac{p}{q}$. For any $v\in V$, the limit $v^*:=\lim_{m\to\infty}f_{mq}(v)$ exists as a point in $V$.
\end{lemma}

\begin{proof} If $f_q(v)=v$, then by Lemma~\ref{lem:cont-prop}(iii) we have $f_{mq}(v)=v$ for all $m$, and hence $v^*=v\in V$.

If $v\neq f_q(v)$, then it follows from the the fact that $G$ is cyclic and from Lemma~\ref{lem:wf-as-limit} that $f_{iq}(v)\neq v$ for all $i\ge 1$. There are two possibilities. If $v\prec f_q(v)\preceq f_{2q}(v)\prec v$, then $G$ cyclic implies that for all $m\ge0$ we have $v\prec f_q(v) \preceq f_{2q}(v) \preceq\ldots\preceq f_{mq}(v)\prec v$. Alternatively, if $v\prec f_{2q}(v)\preceq f_q(v)\prec v$ then we have $v \prec f_{mq}(v) \preceq f_{(m-1)q}(v) \preceq\ldots\preceq f_q(v)\prec v$. In either of these cases, the sequence $f_q(v),f_{2q}(v),f_{3q}(v),\ldots$ is a weakly monotonic sequence in the interval $S^1\setminus\{v\}$ (which can be viewed as a bounded subset of $\R$), and therefore this sequence has a limit point in $S^1$. Moreover, we have $v^*\in V$ since $v^*$ is the limit of a sequence of points $\{f_{mq}(v)\}$ in the closed set $V$.
\end{proof}

\begin{lemma}\label{lem:lim-prop}
 If $G$ is a closed and continuous cyclic graph with $\wf(G)=\frac{p}{q}$, then for all $v\in V$ and $i\ge0$ we have
\begin{enumerate}
\item[(i)] $(f_i(v))^*=f_i(v^*)$, and
\item[(ii)] $f_i(v^*)=v^*$ if and only if $i$ is a multiple of $q$.
\end{enumerate}
\end{lemma}

\begin{proof}
Note that both sides of (i) are well-defined since $f_i(v)\in v$ and since $v^*\in V$. We have that $f_i$ is continuous since $\gamma_i$ is. Hence
\[ (f_i(v))^* = \lim_{m \to \infty} f_{mq+i}(v) = f_i(\lim_{m \to \infty} f_{mq}(v)) = f_i(v^*), \]
giving (i). For (ii), note that if $i=m'q$ is a multiple of $q$, then
\[ f_i(v^*)=f_{m'q}(\lim_{m \to \infty} f_{mq}(v))=\lim_{m \to \infty}f_{(m+m')q}=v^*. \]
The reverse direction of this statement follows from Lemma~\ref{lem:wf-as-limit} and since $\wf(G)=\frac{p}{q}$ with $p$ and $q$ relatively prime.
\end{proof}

\begin{definition}
Let $G$ be a closed and continuous cyclic graph with $\wf(G)=\frac{p}{q}$. We define an equivalence relation on the set of all permanently fast points by declaring two permanently fast points $v$ and $u$ to be equivalent if $v^*=u^*$.
\end{definition}

\begin{definition}
Let $G$ be a closed and continuous cyclic graph with $\wf(G)=\frac{p}{q}$. An invariant set $I=E_0\cup\ldots\cup E_{q-1}$ of permanently fast points is a union of $q$ equivalence classes $E_0,\ldots,E_{q-1}$ of permanently fast points such that for all $v_i\in E_i$ and $j\ge0$, we have $f_j(v_i)\in E_{i+j\md q}$.
\end{definition}

\begin{lemma}\label{lem:invariant-sets}
If $G$ is a closed and continuous cyclic graph with $\wf(G)=\frac{p}{q}$, then the invariant sets of permanently fast points partition the set of all permanently fast points.
\end{lemma}

\begin{proof}
We first show that every permanently fast point $v\in V$ is in some invariant set $I$. For $0\le i \le q-1$, let $E_i$ be the equivalence class of permanently fast points containing $f_i(v)$. To see that $f_1(E_i)\subseteq E_{i+1}$ for $1\le i \le q-2$, note that if $v^*=u^*$, then $(f_1(v))^*=f_1(v^*)=f_1(u^*)=(f_1(u))^*$ by Lemma~\ref{lem:lim-prop}(i). To see that $f_1(E_{q-1})\subseteq E_0$, note $u\in E_{q-1}$ implies $u^*=(f_{q-1}(v))^*=f_{q-1}(v^*)$, and hence $(f_1(u))^*=f_1(u^*)=f_q(v^*)=v^*$ by Lemma~\ref{lem:lim-prop}(ii). To see that these $q$ equivalence classes are disjoint, suppose that $v$ and $f_i(v)$ are in the same equivalence class for some $i>0$, i.e.\ $v^* = (f_i(v))^*$. This gives $v^*=(f_i(v))^*=f_i(v^*)$, which by Lemma~\ref{lem:lim-prop}(ii) means $i$ is a multiple of $q$. It follows that $v$ is in some invariant set $I$.

We next show that every permanently fast point $v$ is in at most one invariant set. Suppose that $I$ and $I'$ are invariant sets of permanently fast points both containing $v$; relabel the $E_i$ and $E'_i$ so that $v\in E_0\subseteq I$ and $v \in E'_0\subseteq I'$. By definition of equivalence classes we have $E_0=E'_0$. Applying $f_i$ gives $f_i(E_0)=f_i(E'_0)$ and hence $E_i=E'_i$.
\end{proof}

\subsection{Closed, continuous, and singular cyclic graphs}\label{ssec:cont-sing}

Let $G$ be a closed and continuous cyclic graph that is singular, i.e.\ $\wf(G)=\frac{l}{2l+1}$ and the supremum is attained. Let $V_\numorb\subseteq V$ be the set of periodic points, and let $\numorb=\frac{|V_\numorb|}{2l+1}$ be the cardinal number of periodic orbits. Let $F$ be the cardinal number of invariant sets of permanently fast points; by Lemma~\ref{lem:invariant-sets} $F$ is equal to the number of equivalence classes of permanently fast points divided by $2l+1$. We restrict attention to the case where $\numorb$ and $F$ are finite.

Let $\tfin(G)$ be the set of all finite $W\subseteq V$ such that 
\begin{itemize}
\item $V_\numorb\subseteq W$,
\item if $v\in W$ is a fast point of $G$ that achieves periodicity (with $f_m(v)$ periodic in $G$), then also $f_i(v)\in W$ for all $1\le i\le m$, and
\item if $I$ is an invariant set of permanently fast points of $G$, then $G[I\cap W]$ contains a single periodic orbit isomorphic to $C^l_{2l+1}$.
\end{itemize}
Note $V_\numorb\subseteq W$ implies $\wf(G[W])=\frac{l}{2l+1}$. Note $G[W]$ has exactly $\numorb+F$ periodic orbits, one for each periodic orbit of $G[V_\numorb]$ and one for each invariant set of permanently fast points of $G$.

We will see that $\tfin(G)$ is nonempty; indeed in Lemma~\ref{lem:singular-wf-cofinal} we show that $\tfin(G)$ is cofinal in $\fin(G)$.

\begin{lemma}\label{lem:singular-wf}
Let $G$ be a closed, continuous, singular cyclic graph with $\wf(G)=\frac{l}{2l+1}$, and let $\numorb$ and $F$ be finite. Then for each $W\in\tfin(G)$ we have $\cl(G[W])\simeq \bigvee^{\numorb+F-1}S^{2l}$. Furthermore, for $W,\widetilde{W}\in\tfin(G)$ with $W\subseteq \widetilde{W}$ the inclusion map $\cl(G[W])\hookrightarrow\cl(G[\widetilde{W}])$ is a homotopy equivalence.
\end{lemma}

\begin{proof}
Since $G[W]$ has $\numorb+F$ periodic orbits, we have $\cl(G[W])\simeq \bigvee^{\numorb+F-1}S^{2l}$ by Proposition~\ref{prop:vrchtpy}.

Let $W,\widetilde{W}\in\tfin(G)$ with $W\subseteq \widetilde{W}$. We claim that each periodic orbit of $G[\widetilde{W}]$ is hit by the inclusion map $W \hookrightarrow \widetilde{W}$. This is clear if the vertices of the periodic orbit are in $V_\numorb$. Alternatively, if the periodic orbit of $G[\widetilde{W}]$ is in an invariant set of permanently fast points of $G$, then it is hit by the periodic orbit of $G[W]$ in the same invariant set. It follows from Proposition~\ref{prop:rank} (with $\Z^{\numorb+F-1}=\redhom_{2l}(\cl(G[W]))\to\redhom_{2l}(\cl(G[W]))=\Z^{\numorb+F-1}$ an isomorphism) that inclusion $\cl(G[W])\hookrightarrow\cl(G[\widetilde{W}])$ is a homotopy equivalence.
\end{proof}

\begin{lemma}\label{lem:singular-wf-cofinal}
Let $G$ be a closed, continuous, and singular cyclic graph with $\numorb$ and $F$ finite. Then the poset $\tfin(G)$ is cofinal in $\fin(G)$.
\end{lemma}

\begin{proof} Let $W\in \fin(G)$ be arbitrary; we must show there exists some $\widetilde{W}\in\tfin(G)$ with $W\subseteq \widetilde{W}$. We will construct $\widetilde{W}$ by adding points to $W$, and we begin by letting $\widetilde{W}=W\cup V_\numorb$. For each $v_0\in W$ that is a fast point of $G$ achieving periodicity, let $m$ be the smallest positive integer such that $f_m(v_0)\in V$ is periodic, and add the collection of points $v_0\diredge v_1\diredge \ldots\diredge v_m=f_m(v_0)$ (with $\sum_{i=0}^{m-1}\ddist(v_i,v_{i+1})=\gamma_m(v_0)$) to $\widetilde{W}$. Hence the only periodic points of $G[\widetilde{W}]$ are either in $V_\numorb$ or in some invariant set of permanently fast points of $G$.

Let $\wf(G)=\frac{l}{2l+1}$. Next we need to ensure that if $I=E_0\cup\ldots\cup E_{2l}$ is an invariant set of permanently fast points for $G$, then $G[\widetilde{W}\cap I]$ contains a single periodic orbit. Add a finite number of points from $I$ to $\widetilde{W}$ to ensure that $G[\widetilde{W}\cap I]$ contains at least one periodic orbit (this is possible by adding $2l+1$ points $v_0\to v_1\to\ldots\to v_{2l}$ in $I$ with $\sum_{i=0}^{2l}\ddist(v_i,v_{i+1\md 2l})=l$). If $G[\widetilde{W}\cap I]$ contains $k\ge 2$ periodic orbits, then we may label them so that the $i$-th periodic orbit for $1\le i\le k$ consists of the points $v_0^i,\ldots,v_{2l}^i$ with $v_j^i\in E_j$ for all $j$, and with $v_j^k\prec v_j^{k-1}\prec\ldots\prec v_j^1\prec v_j^k$. To get rid of the $k$-th periodic orbit, find some directed path $v_0^k=u_0\diredge u_1\diredge\ldots\diredge u_m=v_0^{k-1}$ in $G$ with $u_i\in[v_{i\md 2l+1}^k,v_{i\md 2l+1}^{k-1}]_{S^1}$ for $0\le i\le m$; such a path must exist since $(v_0^{k})^*=(v_0^{k-1})^*$ in $G$. Add the points $\{u_0,\ldots,u_m\}$ to $\widetilde{W}$. It follows that $G[\widetilde{W}\cap I]$ now contains at most $k-1$ periodic orbits; repeat this process until  $G[\widetilde{W}\cap I]$ contains a single orbit. Hence we have found some $\widetilde{W}\in\tfin(G)$ with $W\subseteq \widetilde{W}$.
\end{proof}

\begin{theorem}\label{thm:homotopy-singular-II}
Let $G$ be a closed, continuous, and singular cyclic graph with $\wf(G)=\frac{l}{2l+1}$ and $\numorb$ and $F$ finite. Then $\cl(G)\simeq\bigvee^{\numorb+F-1}S^{2l}$.
\end{theorem}

\begin{proof}
Note we can't have both $\numorb=0$ and $F=0$, since that would imply that every $v\in V$ is a slow point, and hence the supremum in $\wf(G)=\frac{l}{2l+1}$ would not be not attained. By Lemma~\ref{lem:singular-wf-cofinal} we have that $\tfin(G)$ is confinal in $\fin(G)$, and hence 
\[ \cl(G)=\colim_{W \in \tfin(G)} \cl(G[W])\simeq \hocolim_{W \in \tfin(G)}\cl(G[W]). \]
By Lemma~\ref{lem:singular-wf} all maps in the diagram $\cl(G[-])\colon \tfin(G)\to Top$ are homotopy equivalences between spaces homotopy equivalent to $\bigvee^{\numorb+F-1}S^{2l}$, giving $\cl(G)\simeq \hocolim_{W \in \tfin(G)}\cl(G[W])\simeq \bigvee^{\numorb+F-1}S^{2l}$ by the Quasifibration Lemma \cite[Proposition~3.6]{Hocolim}.
\end{proof}

\section{Geometric lemmas for the ellipse}\label{sec:geometric-lemmas}

We next prove several geometric lemmas about ellipses. These lemmas will allow us to describe the infinite cyclic graphs corresponding to Vietoris--Rips complexes of ellipses, and hence the homotopy types of Vietoris--Rips complexes of ellipses (in Section~\ref{sec:VR-of-ellipse}).

Let $d\colon\R^2\times\R^2\to\R$ be the Euclidean metric. Let $Y=\{(x,y)\in \R^2~|~(x/a)^2+y^2=1\}$ be an ellipse with $a>1$. We will later restrict attention to ellipses of small eccentricity, meaning the length of the semi-major axis satisfies $1<a\le\sqrt{2}$. The threshhold $a\le\sqrt{2}$ guarantees that all metric balls in $Y$ are connected (Lemma~\ref{lem:d_p-extrema}). Note we equip the ellipse $Y$ with the Euclidean metric $d$, even though we have been using the geodesic metric on the circle $S^1$.

We will often think of $Y$ as $S^1$, simply by fixing an orientation-preserving homeomorphism $Y\to S^1$. We write $p_1 \preceq p_2 \preceq \ldots \preceq p_n \preceq p_1$ if $p_1,p_2,\ldots,p_n \in Y$ are oriented in a clockwise fashion, allowing equality. We may also replace $p_i \preceq p_{i+1}$ with $p_i \prec p_{i+1}$ if in addition $p_i \neq p_{i+1}$. For $p,q\in Y$ with $p\neq q$ we denote the closed clockwise arc from $p$ to $q$ by $[p,q]_Y=\{z\in Y~|~p \preceq z \preceq q \preceq p\}$. Open and half-open arcs are defined similarly and denoted $(p,q)_Y$, $(p,q]_Y$, and $[p,q)_Y$.

\begin{definition}
Define the continuous function $h\colon Y \to Y$ which maps a point $p \in Y$ to the unique point in the intersection of the normal line to $Y$ at $p$ with $Y\setminus\{p\}$.
\end{definition}

\begin{lemma}\label{lem:h-surjective}
Given a point $p\in Y$ with $p\notin\{(\pm a,0),(0,\pm1)\}$, there is exactly one point $q$ in the diametrically opposite quadrant (on the other side of both axes) satisfying $h(q) = p$. In particular $h\colon Y \to Y$ is surjective.
\end{lemma}

\begin{proof}
Without loss of generality, let $p = (x,y) \in Y$ be in the bottom left quadrant, meaning $x<0$ and $y<0$. We must find some $q = (a \cos t, \sin t)$ with $0 < t < \frac{\pi}{2}$ and $h(q)=p$. The direction of the vector perpendicular to the ellipse at $q$ is $(1, a \tan t)$, and the direction of the vector from $p$ to $q$ is $(a \cos t - x, \sin t - y)$. Thus $h(q) = p$ whenever 
\[ c\begin{pmatrix} a \cos t - x\\ \sin t - y \end{pmatrix} = \begin{pmatrix} 1\\ a \tan t \end{pmatrix} \]
for some $c\in\R$, or alternatively, when
\[ g(t):=(a^2 - 1) \sin t - ax \tan t + y= 0. \]
The derivative $g'(t)=(a^2 - 1) \cos t - ax \sec^2 t$ is strictly positive on $0 < t < \frac{\pi}{2}$ since $a>1$ and $x<0$. Therefore $g(t)$ is strictly increasing on $0 < t < \frac{\pi}{2}$, and so there is at most one $q$ with $h(q) = p$. To see that there is at least one such $q$, note that $g$ is continuous with $g(0)=y<0$ and $\lim_{t\to(\frac{\pi}{2}^-)}g(t)=\infty$.

The fact that $h\colon Y \to Y$ is surjective follows since $h((\pm a,0)) = (\mp a,0)$ and $h((0,\pm 1)) = (0,\mp 1)$, and by the symmetry of the ellipse.
\end{proof}

\begin{lemma}\label{lem:h-opposite-quadrant}
Ellipse $Y=\{(x,y)\in \R^2~|~(x/a)^2+y^2=1\}$ with $a>1$ is an ellipse of small eccentricity (meaning $1<a\le\sqrt{2}$) if and only if for each point $q\in Y$ not on an axis, $h(q)$ is in the quadrant diametrically opposite from $q$.
\end{lemma}

\begin{proof}
We restrict attention, without loss of generality, to the case where $q = (a \cos t, \sin t)$ with $0 < t < \pi / 2$. Let $L$ be the line perpendicular to $Y$ at $q$. The slope of $L$ is $a \tan t$, so the intersection of $L$ with the $x$-axis is the point $(0,y_0)$ where
\[y_0 =\sin t - (a \cos t)(a \tan t) = (1 - a^2) \sin t.\]
For $1 < a \le \sqrt 2$ we have $-1<y_0<0$, and hence the $h(q)$ must lie in the lower left quadrant. For $a>\sqrt{2}$ we see that for $t$ sufficiently close to $\pi/2$, we'll have $y_0<-1$, meaning that point $h(q)$ will be in the lower right quadrant and the property does not hold.
\end{proof}

\begin{lemma}\label{lem:h-bijective}
The function $h$ is bijective if and only if $Y$ is an ellipse of small eccentricity.
\end{lemma}

\begin{proof}
If $Y$ is an ellipse of small eccentricity, then $h$ is bijective by Lemmas~\ref{lem:h-surjective} and \ref{lem:h-opposite-quadrant}. If $Y$ is not an ellipse of small eccentricity, then by Lemmas~\ref{lem:h-opposite-quadrant} and \ref{lem:h-surjective} there is some $p \in Y$ with points $q$ in an adjacent quadrant and $q'$ in the opposite quadrant satisfying $h(q) = p = h(q')$. Hence $h$ is not injective.
\end{proof}

For $p \in Y$, define the function $d_p \colon Y \to \R$ by $d_p(q)=d(p,q)$.

\begin{lemma}\label{lem:d_p-extrema}
For $Y$ an ellipse of small eccentricity and $p \in Y$, the only local extrema of $d_p \colon Y \to \R$ are a global minimum at $p$ and a global maximum at $h^{-1}(p)$. It follows that nonempty metric balls in $Y$ are either contractible or all of $Y$.
\end{lemma}

\begin{proof}
    Clearly $p$ is the global minimum of $d_p$. For any local extrema $q \neq p$ of $d_p$ it must be the case that the circle of radius $d(p,q)$ centered at $p$ is tangent to $Y$ at $q$, and hence $h(q)=p$. By Lemma~\ref{lem:h-bijective}, $h$ is bijective, and since $Y$ is compact it follows that $h^{-1}(p)$ is the global maximum of $d_p$ and that there are no other local extrema besides $p$ and $h^{-1}(p)$.
\end{proof}

Let $Y$ be an ellipse of small eccentricity. If $p\in Y$ and $0<r\le d_p(h^{-1}(p))$, then we define $g_r(p)\in Y$ by
\[ g_r(p)=\underset{q\in [p,h^{-1}(p)]_Y}{\mbox{argmax}}\{d(p,q)~|~d(p,q)\le r\}. \]
Note that $g_r(p)$ is the unique solution $q$ to $d(p,q) = r$ on $[p, h^{-1}(p)]_Y$.

Note $\inf_{p \in Y}\{d(p,h^{-1}(p))\}=2$, and that this infimum is attained for $p=(0,\pm1)$ and $h^{-1}(p)=(0,\mp1)$. Hence for some fixed $0<r<2$ we can define a function $g_r \colon Y \to Y$. It is not hard to see that $g_r$ is bijective and of degree (i.e.\ winding number) one. It follows that if $y_0\prec y_1\prec y_2\prec y_0$ then $g_r(y_0)\prec g_r(y_1)\prec g_r(y_2)\prec g_r(y_0)$ (i.e.\ $g_r$ is cyclic), and if $0<r<r'<2$ then $y\prec g_r(y)\prec g_{r'}(y)\prec y$.

\begin{lemma}\label{lem:g_r-cont-r}
For $Y$ an ellipse of small eccentricity and $p\in Y$, the function $(0,2)\to Y$ defined by $r\mapsto g_r(p)$ is continuous.
\end{lemma}

\begin{proof}
Function $d_p\colon (p, h^{-1}(p))_Y\to \R$ is continuous since $(p, h^{-1}(p))_Y$ is the image of a continuous curve in the plane. Hence given $\epsilon > 0$, there exists $\delta > 0$ such that $d(q,g_r(p)) < \delta$ for $q\in (p, h^{-1}(p))_Y$ implies $|d_p(q) - r| < \epsilon$. Choose $q^- \in (p,g_r(p))_Y$ with $d(q^-,g_r(p))\le\min\{\delta, \epsilon\}$, and choose $q^+ \in (g_r(p), h^{-1}(p))_Y $ with $d(g^+,g_r(p))\le\min\{\delta, \epsilon\}$; this is possible since $(p, h^{-1}(p))_Y$ is a continuous curve through $g_r(p)$. Since $d_p$ is strictly increasing on $(p, h^{-1}(p))_Y$ by Lemma~\ref{lem:d_p-extrema}, we have $d_p(q^-) < r < d_p(q^+)$. Thus for $|r - r'| < \min(r - d_p(q^-),d_p(q^+) - r)$, we see that $g_{r'}(p) \in (q^-, q^+)_Y$ and hence $d(g_r(p), g_{r'}(p)) < \epsilon$.
\end{proof}

\begin{lemma}\label{lem:g_r-cont-p}
For $Y$ an ellipse of small eccentricity and $0<r<2$, the function $g_r\colon Y\to Y$ is continuous.
\end{lemma}

\begin{proof}
Let $\epsilon > 0$ and $p \in Y$. By Lemma~\ref{lem:g_r-cont-r}, there exists some $\delta > 0$ such that $|r-r'| < \delta$ implies $d(g_r(p),g_{r'}(p)) < \epsilon$; choose such a $\delta$ with $\delta<\epsilon$. By the triangle inequality, we see that if $p'$ satisfies $d(p',p)<\delta$, then $d(p', g_{r-\delta}(p)) < r < d(p', g_{r+\delta}(p))$. Hence by continuity of $d_p\colon (p, h^{-1}(p))_Y\to \R$, there is some point $q'\in(g_{r-\delta}(p), g_{r+\delta}(p))_Y$ (which implies $d(q',g_r(p))<\epsilon$) satisfying $d(p',q') = r$. Thus $g_r\colon Y\to Y$ is continuous.
\end{proof}

\begin{lemma}\label{lem:g_r-cont}
For $Y$ an ellipse of small eccentricity, the function $G\colon Y\times(0,2)\to Y$ defined by $G(p,r)=g_r(p)$ is continuous.
\end{lemma}

\begin{proof}
Let $(p,r)\in Y\times(0,2)$. Restrict attention to a small enough open neighborhood $U$ of $(p,r)$, homeomorphic to $\R^2$, so that there exists some $q\in Y$ with $G(U)\subseteq Y\setminus\{q\}$. We may now parametrize $Y\setminus\{q\}$ as a subset of $\R$ and view the function $G|_U$ as a real-valued function on an open subset of the plane. Since $G|_U$ is monotonic in one (in fact both) of its variables, it follows from \cite[Proposition~1]{kruse1969joint}, Lemma~\ref{lem:g_r-cont-r}, and Lemma~\ref{lem:g_r-cont-p} that $G|_U$ is jointly continuous. Since this is true for all $(p,r)\in Y\times(0,2)$, it follows that $G$ is jointly continuous.
\end{proof}

\begin{lemma}\label{lem:unique-inscribed}
Let $Y$ be an ellipse of small eccentricity and let $p \in Y$. Then there exists a unique inscribed equilateral triangle in $Y$ containing $p$ as one of its vertices.
\end{lemma}

\begin{proof}
For the proof of existence, let $p\in Y$ (see Figure~\ref{fig:EquilateralExists}). Let $R_{\pi/3}$ (resp.\ $R_{-\pi/3}$) be the rotation of $\R^2$ by $\pi/3$ radians counterclockwise (resp.\ clockwise) about point $p$. There must be at least one intersection point in $Y \cap R_{\pi/3}(Y)$ other than $p$; call this point $q$. Let $q'=R_{-\pi/3}(q)$ and note that $q'\in Y$. By construction, we see that $d(p,q)=d(p,q')$ and that $\angle qpq' = \pi/3$, so there exists an inscribed equilateral triangle containing the vertices $p,q,q'\in Y$.

\begin{figure}[h]
\begin{center}
\includegraphics[scale=0.4]{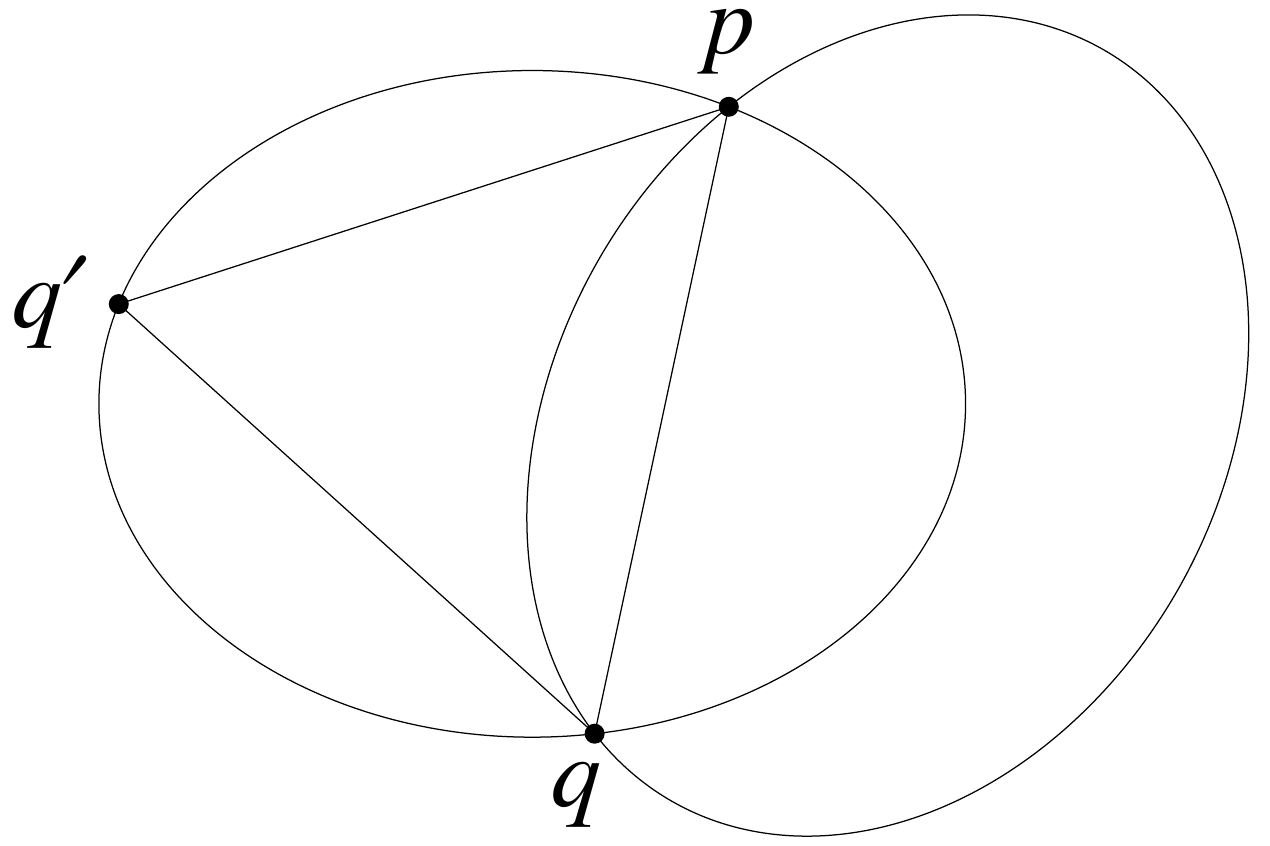}
\end{center}
\caption{Figure for the proof of Lemma~\ref{lem:unique-inscribed}.}
\label{fig:EquilateralExists}
\end{figure}

For uniqueness, suppose for a contradiction that there were two distinct inscribed equilateral triangles in $Y$ with side lengths\footnote{We will learn later in Theorem~\ref{thm:extrema-of-s} that necessarily $r'<2$, but this isn't needed now.} $r<r'$ and vertex sets $\{p,g_r(p),g_r^2(p)\}$ and $\{p,g_{r'}(p),g_{r'}^2(p)\}$. By Lemma~\ref{lem:d_p-extrema} we have $g_r(p),g_{r'}(p)\in(p,h^{-1}(p))_Y$ and $g_r^2(p),g_{r'}^2(p)\in(h^{-1}(p),p)_Y$.
Note $g_r(p)\in(p,g_{r'}(p))_Y$ implies, since $g_r$ is a bijection of degree one, that
\[ g_r^2(p)\in(g_r(p),g_r(g_{r'}(p)))_Y\subseteq(g_r(p),g_{r'}^2(p))_Y. \]
This gives $p\prec h^{-1}(p)\prec g_r^2(p)\prec g_{r'}^2(p)\prec p$. Hence $r=d(g_r^2(p),p)>d(g_{r'}^2(p),p)=r'$, a contradiction.
\end{proof}

\begin{remark}
Our proof of Lemma~\ref{lem:unique-inscribed} relies on the fact that $a\le\sqrt{2}$ and on Lemmas~\ref{lem:h-bijective} and \ref{lem:d_p-extrema}. Simulations suggest that Lemma~\ref{lem:unique-inscribed} is true also for $a\le\sqrt{7/3}$.
\end{remark}

For $Y$ an ellipse of small eccentricity, by Lemma~\ref{lem:unique-inscribed} we may define a continuous function $s \colon Y \to \R$ which maps each $y\in Y$ to the side length of the unique inscribed equilateral triangle containing this vertex. 

\begin{lemma}\label{lem:s-cont}
For $Y$ an ellipse of small eccentricity, the function $s\colon Y\to \R$ is continuous.
\end{lemma}

\begin{proof}
Let $p\in Y$. Since $g_r(p)$ is the unique solution $q$ to $d_p(q) = r$ on $(p, h^{-1}(p))_Y$, and since $d_p$ is strictly increasing along this interval, we see that $g_r(p)$ is monotonic in $r$, that is, for $0 < r < r' < d(p, h^{-1}(p))$ we have $p \prec g_r(p) \prec g_{r'}(p) \prec p$. It follows that for $|r - r'|$ sufficiently small, we have 
\begin{equation}\label{eq:monotonic}
g_{r}(p) \prec g_r^2(p) \prec g_{r'}^2(p) \prec g_{r}(p)\quad\mbox{and}\quad g_r^2(p) \prec g_r^3(p) \prec g_{r'}^3(p) \prec g_r^2(p).
\end{equation}

Let $p \in Y$ and $r = s(p)$. Let $\epsilon>0$ be an arbitrarily small constant satisfying $0<r-\epsilon$ and $r+\epsilon<d(p,h^{-1}(p))$; by \eqref{eq:monotonic} we have $g_r^2(p) \prec g_{r-\epsilon}^3(p) \prec p \prec g_{r+\epsilon}^3(p) \prec g_r^2(p)$. By continuity of $g_r^3\colon Y\to Y$, for $\tilde{p}\in Y$ with $d(p,\tilde{p})$ sufficiently small we have $g_r^2(p) \prec g_{r-\epsilon}^3(\tilde{p}) \prec \tilde{p} \prec g_{r+\epsilon}^3(\tilde{p}) \prec g_r^2(p)$. The monotonicity and continuity of $g_r^3$ then imply there exists some $r^*\in(r-\epsilon,r+\epsilon)$ with $g_{r^*}^3(\tilde{p}) = \tilde{p}$. Hence $s(\tilde{p})=r^*$ with $|s(p)-s(\tilde{p})|=|r - r^*| < \epsilon$, so $s$ is continuous.
\end{proof}

One can verify that
\begin{align*}
\Delta_{(a,0)}&=\Bigl\{(a,0),\Bigl(-\frac{3a-a^3}{a^2+3}, \frac{2\sqrt{3}a}{a^2+3}\Bigr),\Bigl(-\frac{3a-a^3}{a^2+3}, -\frac{2\sqrt{3}a}{a^2+3}\Bigr)\Bigr\} \\
\Delta_{(-a,0)}&=\Bigl\{(-a,0),\Bigl(\frac{3a-a^3}{a^2+3}, \frac{2\sqrt{3}a}{a^2+3}\Bigr),\Bigl(\frac{3a-a^3}{a^2+3}, -\frac{2\sqrt{3}a}{a^2+3}\Bigr)\Bigr\}
\end{align*}
are inscribed equilateral triangles of side length $r_1=s((\pm a,0))=\frac{4\sqrt{3}a}{a^2+3}$, and that
\begin{align*}
\Delta_{(0,1)}&=\Bigl\{(0,1),\Bigl(\frac{2\sqrt{3}a^2}{3a^2+1}, -\frac{3a^2-1}{3a^2+1}\Bigr),\Bigl(-\frac{2\sqrt{3}a^2}{3a^2+1}, -\frac{3a^2-1}{3a^2+1}\Bigr)\Bigr\} \\
\Delta_{(0,-1)}&=\Bigl\{(0,-1),\Bigl(\frac{2\sqrt{3}a^2}{3a^2+1}, \frac{3a^2-1}{3a^2+1}\Bigr),\Bigl(-\frac{2\sqrt{3}a^2}{3a^2+1}, \frac{3a^2-1}{3a^2+1}\Bigr)\Bigr\}
\end{align*}
are inscribed equilateral triangles of side length $r_2=s((0,\pm 1))=\frac{4\sqrt{3}a^2}{3a^2+1}$; see Figure~\ref{fig:r1r2}. Furthermore we have $r_1 < r_2 < 2$.

\begin{figure}[h]
\begin{center}
\includegraphics[scale=0.6]{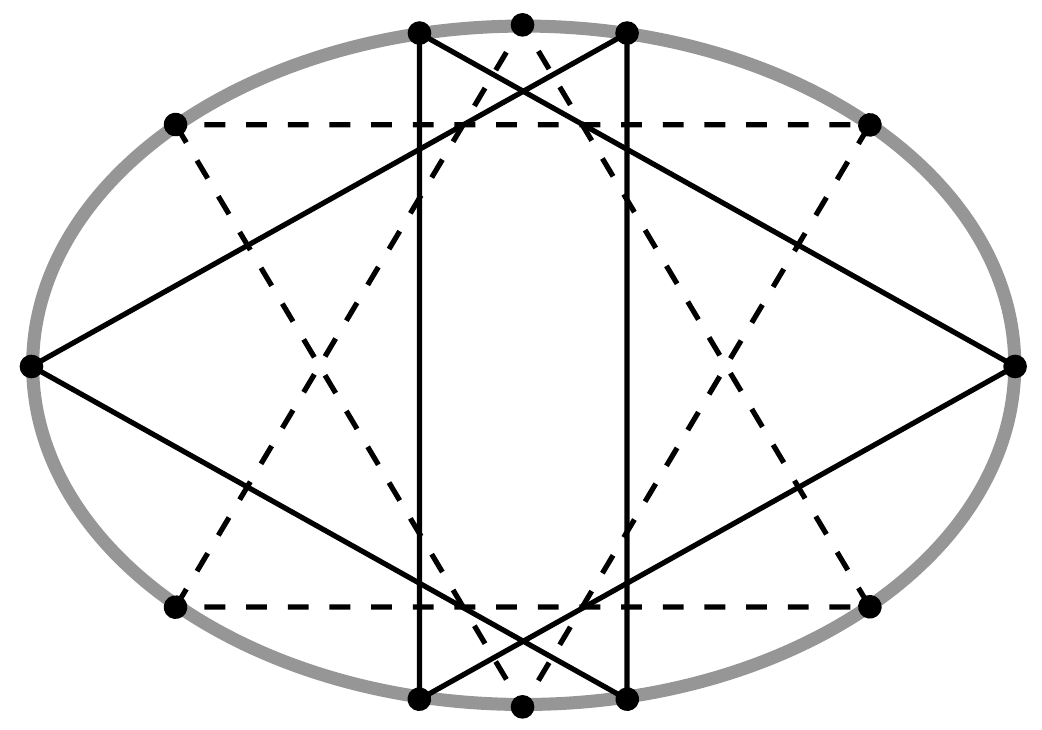}
\end{center}
\caption{The inscribed equilateral triangles $\Delta_{(\pm a,0)}$ drawn with solid lines have radius $r_1$, and the inscribed equilateral triangles $\Delta_{(0,\pm 1)}$ drawn with dashed lines have radius $r_2$.}
\label{fig:r1r2}
\end{figure}

\begin{theorem}\label{thm:extrema-of-s}
For $Y$ an ellipse of small eccentricity, the six vertices in the two inscribed equilateral triangles $\Delta_{(\pm a,0)}$ are global minima of $s\colon Y \to \R$, and the six vertices in the two inscribed equilateral triangles $\Delta_{(0,\pm 1)}$ are global maxima. There are no other local extrema of $s$.
\end{theorem}

\begin{proof}
We first show that given any $r>0$, there are at most twelve points $p\in Y$ satisfying $s(p)=r$. Let $(x_1,y_1),(x_2,y_2)\in Y$ be two points in the ellipse at distance $r$ apart. We consider all points $(x,y)\in\R^2$ at distance $r$ from each of $(x_1,y_1)$ and $(x_2,y_2)$. The triangle $\{(x_1,y_1),(x_2,y_2),(x,y)\}$ is then an inscribed equilateral triangle in $Y$ if and only if $(x,y)\in Y$.

Consider the following system of polynomial equations in the variables $x_1,y_1,x_2,y_2,x,y,a,r\in\R$:
\begin{align}\label{eq:system}
(x_1/a)^2+y_1^2-1&=0 &(x_1-x_2)^2+(y_1-y_2)^2-r^2=0\nonumber\\
(x_2/a)^2+y_2^2-1&=0 &(x-x_1)^2+(y-y_1)^2-r^2=0\\
(x/a)^2+y^2-1&=0 &(x-x_2)^2+(y-y_2)^2-r^2=0\nonumber
\end{align}
The three equations on the left encode the requirement that $(x_1,y_1)$, $(x_2,y_2)$, and $(x,y)$ are points in $Y$, and the three equations on the right encode the requirement that $\{(x_1,y_1),(x_2,y_2),(x,y)\}$ is an equilateral triangle. One can use elimination theory to eliminate the variables $x_1,y_1,x_2,y_2,y$ from \eqref{eq:system}, producing a polynomial $p(x,a,r)$ such that any solution to \eqref{eq:system} is also a root of $p$. We use the mathematics software system Sage to compute
\begin{align*} p(x,a,r)&=12288a^{12}r^4 - 6912a^{12}r^6 - 10752a^{10}r^6 + 13568a^8r^6 + 1296a^{12}r^8 \\
&+ 4032a^{10}r^8 - 2208a^8r^8 - 6720a^6r^8 + 3600a^4r^8 - 81a^{12}r^{10} \\
&- 378a^{10}r^{10} - 63a^8r^{10} + 1044a^6r^{10} - 63a^4r^{10} - 378a^2r^{10} - 81r^{10} \\
&- 36864a^{10}r^4x^2 + 36864a^8r^4x^2 + 13824a^{10}r^6x^2 - 4608a^8r^6x^2 \\
&+ 16896a^6r^6x^2 - 26112a^4r^6x^2 - 1296a^{10}r^8x^2 - 432a^8r^8x^2 \\
&-4512a^6r^8x^2 + 4512a^4r^8x^2 + 432a^2r^8x^2 + 1296r^8x^2 + 36864a^8r^4x^4 \\
&- 73728a^6r^4x^4 + 36864a^4r^4x^4 - 6912a^8r^6x^4 + 15360a^6r^6x^4 \\
&- 16896a^4r^6x^4 + 15360a^2r^6x^4 - 6912r^6x^4 - 12288a^6r^4x^6\\
& + 36864a^4r^4x^6 - 36864a^2r^4x^6 + 12288r^4x^6.
\end{align*}
Note $p$ is a polynomial in $x$ of degree 6. It follows that given any $a$ and $r$, a solution to the system of equations \eqref{eq:system} can contain at most 6 distinct values for $x$, and hence at most 12 distinct points $p=(x,y)\in Y$ satisfying $s(p)=r$.

We next show that the extrema of $s\colon Y\to\R$ are as claimed. Symmetry of $s$ about the horizontal and vertical axes, along with the fact that the cardinality of each set $s^{-1}(r)$ is finite, 
gives that the points $(\pm a,0),(0,\pm 1)$ are local extrema of $s$. Since $s(p)=s(g_{s(p)}(p))=s(g_{s(p)}^2(p))$, it follows that all twelve vertices in the triangles $\Delta_{(\pm a,0)},\Delta_{(0,\pm 1)}$ are local extrema. Since the six vertices of $\Delta_{(\pm a,0)}$ are interleaved with the six vertices of $\Delta_{(0,\pm 1)}$, it follows from the intermediate value theorem that for all $r_1<r<r_2$ we have $|s^{-1}(r)|\ge12$. Since $|s^{-1}(r)|\le12$ for all $r>0$, it follows that $s$ is either monotonically increasing or monotonically decreasing between any two adjacent vertices of the triangles $\Delta_{(\pm a,0)}$ and $\Delta_{(0,\pm 1)}$. Hence the vertices of $\Delta_{(\pm a,0)}$ are global minima, the vertices of $\Delta_{(0,\pm 1)}$ are global maxima, and there are no other local extrema of $s$.
\end{proof}

\section{Vietoris--Rips complexes of ellipses of small eccentricity}\label{sec:VR-of-ellipse}\label{sec:VR}

We restrict attention in this section to $Y=\{(x,y)\in \R^2~|~(x/a)^2+y^2=1\}$ an ellipse of small eccentricity, meaning $1<a\le\sqrt{2}$. For $X \subseteq Y$ and $0<r<2$, we orient the edges of graph $\vrless{X}{r}$ (resp.\ $\vrleq{X}{r}$) in a clockwise fashion by specifying that $p \diredge p'$ is a directed edge if $p' \in (p,g_r(p))_Y$ (resp.\ $p' \in (p,g_r(p)]_Y$).

\begin{lemma}\label{lem:ellipse-cyclic}
Let $Y$ be an ellipse of small eccentricity, let $X \subseteq Y$, and let $0<r<2$. Then $\vrless{X}{r}$ and $\vrleq{X}{r}$ are cyclic graphs.
\end{lemma}

\begin{proof}
We can think of the vertex set $X$ as a subset of $S^1$, simply by fixing an orientation-preserving homeomorphism $Y\to S^1$. If $v\diredge u$ is a directed edge of $\vrleq{X}{r}$, then $u\in(v,g_r(v)]_Y$. If $w\in X$ satisfies $v\prec w\prec u\prec v$, then we have $w\in(v,g_r(v)]_Y$ by definition and $u\in(w,g_r(w)]_Y$ since $g_r$ is cyclic. Hence we have directed edges $v\diredge w$ and $w\diredge u$ in $\vrleq{X}{r}$, and therefore $\vrleq{X}{r}$ is a cyclic graph. The proof that $\vrless{X}{r}$ is a cyclic graph is identical after replacing each half-open interval with an open interval.
\end{proof}

For $X\subseteq Y$ finite it follows from Proposition~\ref{prop:vrchtpy} that simplicial complex $\vrc{X}{r}$ is homotopy equivalent to an odd sphere or a wedge of even spheres of the same dimension.

The following notation will prove convenient. Given a point $p\in Y$, let $p'=g_{s(p)}(p)$ and let $p''=g_{s(p)}^2(p)$. Fix $\frac{4\sqrt{3}a}{a^2+3}=r_1<r<r_2=\frac{4\sqrt{3}a^2}{3a^2+1}$. By Theorem~\ref{thm:extrema-of-s} and the intermediate value theorem, there are cyclically ordered points
\begin{equation}\label{eq-z}
z_0\prec z_1\prec z_2\prec z_3\prec z'_0\prec z'_1\prec z'_2\prec z'_3\prec z''_0\prec z''_1\prec z''_2\prec z''_3\prec z_0
\end{equation}
in $Y$ such that
\begin{itemize}
\item $s(z_i)=s(z'_i)=s(z''_i)=r$ for $i=0,1,2,3$.
\item $I_0=(z_0,z_1)_Y\cup(z'_0,z'_1)_Y\cup(z''_0,z''_1)_Y$ and $I_2=(z_2,z_3)_Y\cup(z'_2,z'_3)_Y\cup(z''_2,z''_3)_Y$ are each invariant sets of three equivalence classes of fast points for dynamics $g_r:Y\to Y$, meaning that $s(p)<r$ for all points $p$ in these intervals.
\item $I_1=(z_1,z_2)_Y\cup(z'_1,z'_2)_Y\cup(z''_1,z''_2)_Y$ and $I_3=(z_3,z'_0)_Y\cup(z'_3,z''_0)_Y\cup(z''_3,z_0)_Y$ consist entirely of slow points for dynamics $g_r:Y\to Y$, meaning that $s(p)>r$ for all points $p$ in these intervals.
\end{itemize}

We say a subset $X\subseteq Y$ is $\epsilon$-dense if for each point $y\in Y$, there exists some point $x\in X$ with Euclidean distance $d(x,y)\le\epsilon$.

\begin{theorem}
Let $Y$ be an ellipse of small eccentricity. For any sufficiently dense finite sample $X \subseteq Y$, graph $\vr{X}{r}$ for $r_1 < r < r_2$ has at least two periodic orbits of length three, and hence $\vrc{X}{r}\simeq\bigvee^k S^2$ for some $k \ge 1$.
\end{theorem}

\begin{proof}
We cyclically order the elements of $X$ as $p_0\prec p_1 \prec \ldots \prec p_{n-1}\prec p_0$, where all arithmetic operations on vertex indices are performed modulo $n$. Recall that associated to cyclic graph $\vr{X}{r}$ we have maps $f_m\colon X\to S^1$ from Definition~\ref{def:gamma_m}. Since $f_m$ converges pointwise to $g_r^m$ as the density increases, for $X$ sufficiently dense the interval $(z_0,z_2)_Y$ contains points $p_i \in (z_0, z_1)_Y\cap X$, and $p_j \in (z_1, z_2)_Y\cap X$ such that 
\begin{equation}\label{eq:ordering}
p_i \prec f_3(p_i) \preceq f_3(p_j) \prec p_j \prec p_i.
\end{equation}
Now, suppose for a contradiction that there is no point $x\in [p_i,p_j]_Y\cap X$ with $x=f_3(x)$. Since $f_3$ weakly preserves the cyclic ordering we have $p_i \prec f_3(p_i) \preceq f_3(p_{i+1}) \preceq f_3(p_j) \prec p_j \prec p_i$, and since $p_{i+1}\neq f_3(p_{i+1})$ this gives
\[p_{i+1} \prec f_3(p_{i+1}) \preceq f_3(p_j) \prec p_j \prec p_{i+1}.\]
Note this is \eqref{eq:ordering} with $i$ replaced by $i+1$. Iterating this process gives $p_{j-1} \prec f_3(p_{j-1}) \preceq f_3(p_j) \prec p_j \prec p_{j-1}$, contradicting the fact that $(p_{j-1},p_j)_Y \cap X=\emptyset$. Hence there must be some $x\in [p_i,p_j]_Y \cap X$ with $x=f_3(x)$. It follows also that $\wf(\vr{X}{r})=\frac{1}{3}$.

Repeating this argument with $(z_0,z_2)_Y$ replaced by any of the intervals
\[ (z_0,z_2)_Y, (z_2,z'_0)_Y, (z'_0,z'_2)_Y, (z'_2,z''_0)_Y, (z''_0,z''_2)_Y, \mbox{ or } (z''_2,z_0)_Y \]
gives 6 periodic points, and hence at least two periodic orbits of length three. The statement about the homotopy type of $\vrc{X}{r}$ follows from Proposition~\ref{prop:vrchtpy}.
\end{proof}

\begin{theorem}\label{thm:exactly-m-orbits}
Let $Y$ be an ellipse of small eccentricity. For any $r_1 < r < r_2$, $\epsilon > 0$, and $k \geq 2$, there exists an $\epsilon$-dense finite subset $X \subseteq Y$ so that graph $\vr{X}{r}$ has exactly $k$ periodic orbits of length three, and hence $\vrc{X}{r}\simeq\bigvee^{k-1}S^2$.
\end{theorem}

\begin{proof}
We begin with some preliminaries. Given $q,p\in Y$, we define $d^*(q,p)=\min\{d(q,p),d(q',p'),d(q'',p'')\}$. If $q,p \in Y$ are sufficiently close, then by the continuity of $p\mapsto g_{s(p)}(p)$ (Lemmas~\ref{lem:g_r-cont} and \ref{lem:s-cont}) we have that $d^*(q,p)<\epsilon$. Furthermore, recall the points $z_i,z'_i,z''_i$ in \eqref{eq-z} which depend on $r$. Fix a (necessarily fast) point $p \in (z_0, z_1)_Y$, which implies $p'\in(p,g_r(p))_Y$. By continuity of $g_r\colon Y\to Y$, for $q \in (z_0,p)_Y$ sufficiently close to $p$ we have
\begin{equation}\label{eq:g_r(q)}
g_r(q)\in(p',g_r(p))_Y.
\end{equation}

We will prove the theorem by showing that for any $n,n'\ge1$ with $n+n'=k$, one can produce an $\epsilon$-dense subset $X$ of $Y\setminus\{z_i,z'_i,z''_i~|~i=0,1,2,3\}$ such that $I_0\cap X$ contains $n$ periodic orbits, $I_2\cap X$ contains $n'$ periodic orbits, and $I_1\cap X$ and $I_3\cap X$ contain no periodic orbits for $\vr{X}{r}$.

Pick a point $p_1\in (z_0,z_1)_Y$ such that $d^*(p_1,z_1)<\epsilon$. Add the points $p_1, p'_1, p''_1$ to $X$; these points will form our first periodic orbit in $I_0\cap X$. To see this, note that since $p_1$ is a fast point, $s(p_1) < r$, giving $g_r(p_1)\in(g_{s(p_1)}(p_1),g_r(z_1))_Y=(p'_1,z'_1)_Y$. Hence if we don't add any points in the intervals $(p'_1, g_r(p_1)]_Y$, $(p''_1, g_r(p'_1)]_Y$, or $(p_1, g_r(p''_1)]_Y$, we'll have that $f_r(p_1) = p'_1$, $f_r(p'_1) = p''_1$, and $f_r(p''_1) = p_1$, i.e.\ a new periodic orbit of length 3. Next pick a point $p_2$ satisfying $p_2\in (g_r(p''_1),z_1)_Y$, $p_2'\in (g_r(p_1),z'_1)_Y$, and $p''_2\in (g_r(p_1'),z''_1)_Y$. Add the points $p_2$, $p'_2$, $p''_2$ to $X$ to form our second periodic orbit. Iterating, for $2\le i\le n$ we pick a point $p_i$ satisfying $p_i\in (g_r(p''_{i-1}),z_1)_Y$, $p'_i\in (g_r(p_{i-1}),z'_1)_Y$, and $p''_i\in (g_r(p'_{i-1}),z''_1)_Y$, and we add the points $p_i$, $p'_i$, $p''_i$ to $X$ to form our $i$-th periodic orbit. This creates $n$ periodic orbits in $I_0\cap X$.

We now create $\epsilon$-density in $I_0$. Let $u \in (z_0,z_1)_Y$ satisfy $d^*(z_0,u)<\epsilon$. By compactness of $[u,p_1]_Y\cup[u',p'_1]_Y\cup[u'',p''_1]_Y$, there exists some $\delta' > 0$ such that for all $p \in [u,p_1]_Y$ and $q \in (z_0,p)_Y$ with $d(q,p) < \delta'$, we have both $d^*(q,p)<\epsilon$ and \eqref{eq:g_r(q)}. Let $\delta  = \min (\delta', \epsilon) / 2$. Define $q_1\in(z_0,p_1)_Y$ to be the unique point with $d^*(q_1,p_1)=\delta$; add $q_1$, $q'_1$, $q''_1$ to $X$. Note that this doesn't create any periodic orbits since \eqref{eq:g_r(q)} implies $f_r(q_1)\neq q'_1$, $f_r(q'_1)\neq q''_1$, and $f_r(q''_1)\neq q_1$. Inductively, for $n>1$ define $q_n\in(z_0,q_{n-1})_Y$ to be the unique point with $d^*(q_n,q_{n-1})=\delta$. Iteratively add $q_n, q'_n, q''_n$ to $X$ until we obtain $q_n \in (z_0,u)_Y$, $q'_n \in (z'_0,u')_Y$, and $q''_n \in (z''_0,u'')_Y$ (these three events happen simultaneously, and at some finite stage since $\delta>0$). This does not create any additional periodic orbits, and we now have $\epsilon$-density in $I_0$.

Create $n'$ periodic orbits in an $\epsilon$-dense subset of $I_2$ via an analogous procedure.

Finally, we make $X$ $\epsilon$-dense by adding otherwise arbitrary points to $I_1$ and $I_3$. These new points cannot create any new periodic orbits for $f_r$ because $I_1$ and $I_3$ are sets of slow points for $g_r\colon Y\to Y$. This concludes the construction of an $\epsilon$-dense subset $X$ so that $\vr{X}{r}$ has exactly $k$ periodic orbits. The homotopy type of $\vrc{X}{r}$ follows from Proposition~\ref{prop:vrchtpy} since $\wf(\vr{X}{r})=\frac{1}{3}$.
\end{proof}

In the infinite cyclic graph $\vr{Y}{r}$, one can check that for $m\ge1$, the function $f_m\colon Y\to Y$ in Definition~\ref{def:gamma_m} is equal to $f_m=g_r^m$. Since $g_r$ is continuous by Lemma~\ref{lem:g_r-cont-p}, it follows that each $\gamma_m\colon Y\to\R$ is continuous. Hence $\vr{Y}{r}$ is a closed and continuous cyclic graph.

\begin{theorem}
\label{thm:ellipse-homotopy}
Let $Y$ be an ellipse of small eccentricity, and let $r_1=\frac{4\sqrt{3}a}{a^2+3}$ and $r_2=\frac{4\sqrt{3}a^2}{3a^2+1}$. Then
\[
\vrcless{Y}{r}\simeq \begin{cases}
S^1&\mbox{for }0<r\le r_1\\
S^2&\mbox{for }r_1<r\le r_2
\end{cases}
\ \mbox{and}\ 
\vrcleq{Y}{r}\simeq \begin{cases}
S^1&\mbox{for }0<r<r_1\\
S^2&\mbox{for }r=r_1\\
\bigvee^5 S^2&\mbox{for }r_1<r<r_2\\
\bigvee^3 S^2&\mbox{for }r=r_2.
\end{cases}
\]
Furthermore,
\begin{itemize}
\item For $0<r<\tilde{r}\le r_1$ or $r_1<r<\tilde{r}\le r_2$, inclusion $\vrcless{Y}{r}\hookrightarrow\vrcless{Y}{\tilde{r}}$ is a homotopy equivalence.
\item For $0<r<\tilde{r}< r_1$, inclusion  $\vrcleq{Y}{r}\hookrightarrow\vrcleq{Y}{\tilde{r}}$ is a homotopy equivalence.
\item For $r_1\le r<\tilde{r}\le r_2$, inclusion $\vrcleq{Y}{r}\hookrightarrow\vrcleq{Y}{\tilde{r}}$ induces a rank 1 map on 2-dimensional homology $H_2(-;\F)$ for any field $\F$.
\end{itemize}
\end{theorem}

\begin{proof}
We first prove the homotopy types of $\vrcless{Y}{r}$. For $0<r<r_1$ (resp.\ $r=r_1$) we have $0<\wf(\vrless{Y}{r})<\frac{1}{3}$ (resp.\ $\wf(\vrless{Y}{r})=\frac{1}{3}$ but the supremum is not attained), hence $\vrcless{Y}{r}\simeq S^1$ by Theorem~\ref{thm:homotopy-generic}. For $0<r<\tilde{r}\le r_1$, Theorem~\ref{thm:homotopy-generic} also implies $\vrcless{Y}{r}\hookrightarrow\vrcless{Y}{\tilde{r}}$ is a homotopy equivalence.

Let $r_1<r\le r_2$. By Theorem~\ref{thm:extrema-of-s} and the intermediate value theorem there are cyclically ordered points
\begin{align*}
z_0(r)\prec z_1(r)\prec z_2(r)\prec z_3(r)&\prec z'_0(r)\prec z'_1(r)\prec z'_2(r)\prec z'_3(r)\\
&\prec z''_0(r)\prec z''_1(r)\prec z''_2(r)\prec z''_3(r)\prec z_0(r)
\end{align*}
in $Y$ such that
\begin{align*}
I_0(r)_<&=(z_0(r),z_1(r))_Y\cup(z'_0(r),z'_1(r))_Y\cup(z''_0(r),z''_1(r))_Y\quad\mbox{and}\\
I_2(r)_<&=(z_2(r),z_3(r))_Y\cup(z'_2(r),z'_3(r))_Y\cup(z''_2(r),z''_3(r))_Y
\end{align*}
are each invariant sets of permanently fast points for $\vrless{Y}{r}$, and
\begin{align*}
I_1(r)_<&=[z_1(r),z_2(r)]_Y\cup[z'_1(r),z'_2(r)]_Y\cup[z''_1(r),z''_2(r)]_Y\quad\mbox{and}\\
I_3(r)_<&=[z_3(r),z'_0(r)]_Y\cup[z'_3(r),z''_0(r)]_Y\cup[z''_3(r),z_0(r)]_Y
\end{align*}
are each sets of slow points for $\vrless{Y}{r}$. Hence $\wf(\vrless{Y}{r})=\frac{1}{3}$, the supremum is attained, and $\vrless{Y}{r}$ has $\numorb=0$ periodic orbits and $F=2$ invariant sets of permanently fast points. By Theorem~\ref{thm:homotopy-singular-II} we have $\vrcless{Y}{r}=\cl(\vrless{Y}{r})\simeq\bigvee^{\numorb+F-1}S^2=S^2$.

We now prove that for $r_1<r<\tilde{r}\le r_2$, the inclusion $\vrcless{Y}{r}\hookrightarrow\vrcless{Y}{\tilde{r}}$ is a homotopy equivalence. Fix a finite subset $W\in\tfin(\vrless{Y}{r})$ (notation defined in Section~\ref{ssec:cont-sing}); by Lemma~\ref{lem:singular-wf-cofinal} we can also choose some $\widetilde{W}\supseteq W$ with $\widetilde{W}\in\tfin(\vrless{Y}{\tilde{r}})$. We have the following commutative diagram.
\[ \xymatrix{
\vrcless{W}{r} \ar@{->}[r]^{\iota_r} \ar@{->}[d]^{\iota_W} & \vrcless{Y}{r} \ar@{->}[d]^{\iota_Y}\\
\vrcless{\widetilde{W}}{\tilde{r}} \ar@{->}[r]^{\iota_{\tilde{r}}} & \vrcless{Y}{\tilde{r}}
} \]
The horizontal inclusions $\iota_r$ and $\iota_{\tilde{r}}$ are homotopy equivalences by the proof of Theorem~\ref{thm:homotopy-singular-II}. Using the notation of Proposition~\ref{prop:rank}, we have $\hit(\iota_W)=2$ since the periodic orbit for $\vrless{W}{r}$ in $I_0(r)_<$ hits the periodic orbit for $\vrless{\widetilde{W}}{\tilde{r}}$ in $I_0(\tilde{r})_<$ (note $I_0(r)_<\subseteq I_0(\tilde{r})_<$), and similarly for the periodic orbits in $I_2(r)_<$ and $I_2(\tilde{r})_<$. It follows from Proposition~\ref{prop:rank} that $\iota_W$ is a homotopy equivalence, and hence $\iota_Y$ is a homotopy equivalence.

We next prove the homotopy types of $\vrcleq{Y}{r}$. For $0<r<r_1$ we have $0<\wf(\vrleq{Y}{r})<\frac{1}{3}$, and hence $\vrcleq{Y}{r}\simeq S^1$ by Theorem~\ref{thm:homotopy-generic}. For $0<r<\tilde{r}< r_1$, Theorem~\ref{thm:homotopy-generic} also implies $\vrcleq{Y}{r}\hookrightarrow\vrcleq{Y}{\tilde{r}}$ is a homotopy equivalence.

Note that the cyclic graph $\vrleq{Y}{r_1}$ has winding fraction $\frac{1}{3}$, $P=2$ periodic orbits given by the vertices of $\Delta_{(\pm a,0)}$, and no fast points. It follows from Theorem~\ref{thm:homotopy-singular-II} that $\vrcleq{Y}{r_1}\simeq S^2$.

Let $r_1 < r< r_2$. By Theorem~\ref{thm:extrema-of-s} and the intermediate value theorem the set $\{z_i(r),z'_i(r),z''_i(r)\}$ is a periodic orbit of $\vrleq{Y}{r}$ of length three for $i=0,1,2,3$,
\begin{align*}
I_0(r)_\leq&=(z_0(r),z_1(r))_Y\cup(z'_0(r),z'_1(r))_Y\cup(z''_0(r),z''_1(r))_Y\quad\mbox{and}\\
I_2(r)_\leq&=(z_2(r),z_3(r))_Y\cup(z'_2(r),z'_3(r))_Y\cup(z''_2(r),z''_3(r))_Y
\end{align*} are each invariant sets of permanently fast points for $\vrleq{Y}{r}$, and
\begin{align*}
I_1(r)_\leq&=(z_1(r),z_2(r))_Y\cup(z'_1(r),z'_2(r))_Y\cup(z''_1(r),z''_2(r))_Y\quad\mbox{and}\\
I_3(r)_\leq&=(z_3(r),z'_0(r))_Y\cup(z'_3(r),z''_0(r))_Y\cup(z''_3(r),z_0(r))_Y
\end{align*}
are each sets of slow points for $\vrleq{Y}{r}$. Hence $\wf(\vrleq{Y}{r})=\frac{1}{3}$, the supremum is attained, and $\vrleq{Y}{r}$ has $\numorb=4$ periodic orbits and $F=2$ invariant sets of permanently fast points. By Theorem~\ref{thm:homotopy-singular-II} we have $\vrcleq{Y}{r}=\cl(\vrleq{Y}{r})\simeq\bigvee^{\numorb+F-1}S^2=\bigvee^5 S^2$.

The cyclic graph $\vrleq{Y}{r_2}$ has $\wf(\vrleq{Y}{r_2})=\frac{1}{3}$ and $P=2$ periodic orbits given by the vertices of $\Delta_{(0,\pm 1)}$. The remaining points of $Y$ are partitioned into $F=2$ invariant sets of permanently fast points. It follows from Theorem~\ref{thm:homotopy-singular-II} that $\vrcleq{Y}{r_2}\simeq \bigvee^3 S^2$.

We now prove that for $r_1<r<\tilde{r}<r_2$, the inclusion $\vrcleq{Y}{r}\hookrightarrow\vrcleq{Y}{\tilde{r}}$ induces a rank 1 map on 2-dimensional homology $H_2(-;\F)$ for any field $F$. Fix a finite subset $W\in\tfin(\vrleq{Y}{r})$; by Lemma~\ref{lem:singular-wf-cofinal} we can also choose some $\widetilde{W}\supseteq W$ with $\widetilde{W}\in\tfin(\vrleq{Y}{\tilde{r}})$. We have the following commutative diagrams.
\[ 
\xymatrix{
\vrcleq{W}{r} \ar@{->}[r]^{\iota_r} \ar@{->}[d]^{\iota_W} & \vrcleq{Y}{r} \ar@{->}[d]^{\iota_Y}\\
\vrcleq{\widetilde{W}}{\tilde{r}} \ar@{->}[r]^{\iota_{\tilde{r}}} & \vrcleq{Y}{\tilde{r}}
}
\hspace{3mm} 
\xymatrix{
H_2(\vrcleq{W}{r};\F) \ar@{->}[r]^{\iota_r^*} \ar@{->}[d]^{\iota_W^*} & H_2(\vrcleq{Y}{r};\F) \ar@{->}[d]^{\iota_Y^*}\\
H_2(\vrcleq{\widetilde{W}}{\tilde{r}};\F) \ar@{->}[r]^{\iota_{\tilde{r}}^*} & H_2(\vrcleq{Y}{\tilde{r}};\F)
}
\]
The horizontal maps $\iota_r^*$ and $\iota_{\tilde{r}}^*$ are isomorphisms since the inclusions $\iota_r$ and $\iota_{\tilde{r}}$ are homotopy equivalences by the proof of Theorem~\ref{thm:homotopy-singular-II}. 

We next claim that $\hit(\iota_W)=2$. Note that $\vrleq{W}{r}$ has 6 periodic orbits: the 4 periodic orbits containing $z_0(r)$, $z_1(r)$, $z_2(r)$, $z_3(r)$, and the two periodic orbits in $I_0(r)_\leq$ and $I_2(r)_\leq$. The same is true for $\vrleq{\widetilde{W}}{\tilde{r}}$, with $r$ replaced everywhere by $\tilde{r}$. Note the periodic orbits in $W$ corresponding to $z_0(r)$, $z_1(r)$, and $I_0(r)_\leq$ map under $\iota_W$ to the periodic orbit corresponding to $I_0(\tilde{r})_\leq$ in $\widetilde{W}$, and the periodic orbits in $W$ corresponding to $z_2(r)$, $z_3(r)$, and $I_2(r)_\leq$ map under $\iota_W$ to the periodic orbit corresponding to $I_2(\tilde{r})_\leq$ in $\widetilde{W}$. So $\hit(\iota_W)=2$. Proposition~\ref{prop:rank} now implies $\rank(\iota_W^*)=\hit(\iota_W)-1=1$, and it follows that $\rank(\iota_Y^*)=1$. 

The proof for the case $r_1\le r<\tilde{r}\le r_2$ is analogous.
\end{proof}

Since Theorem~\ref{thm:ellipse-homotopy} gives the homotopy types of $\vrc{Y}{r}$ along with the behavior of inclusions $\vrc{Y}{r}\hookrightarrow\vrc{Y}{\tilde{r}}$, it also allows us to describe the \emph{persistent homology} of $\vrc{Y}{r}$ \cite{EdelsbrunnerHarer,ChazalDeSilvaOudot2013}. Our convention here is that we don't include every point on the diagonal with infinite multiplicity; this allows us to more easily describe the ephemeral summands \cite{ChazalCrawley-BoeveyDeSilva} in the persistent homology of $\vrcleq{Y}{r}$. The following corollary is visualized in Figure~\ref{fig:PH}.

\begin{corollary}\label{cor:ellipse-persistent-homology}
Let $Y$ be an ellipse of small eccentricity, and fix an arbitrary field of coefficients for homology. Then the 1-dimensional persistent homology of $\vrcless{Y}{r}$ (resp.\ $\vrcleq{S^1}{r}$) consists of a single interval $(0,r_1]$ (resp.\ $[0,r_1)$).

The 2-dimensional persistent homology of $\vrcless{Y}{r}$ consists of a single interval $(r_1,r_2]$. The 2-dimensional persistent homology of $\vrcleq{S^1}{r}$ consists of the interval $[r_1,r_2]$, as well as a point $[r,r]$ on the diagonal with multiplicity four for every $r_1<r<r_2$.
\end{corollary}

\begin{proof}
It suffices to show that $\wf(\vr{Y}{r})>\frac{1}{3}$ for $r>r_2$. Indeed, Theorem~\ref{thm:ellipse-homotopy} then gives the correct persistent homology up until scale parameter $r_2$, and Theorems~\ref{thm:homotopy-generic} and \ref{thm:homotopy-singular-I} show that beyond scale $r_2$ there is no homology in dimensions 1 or 2.

Let $r_2<r<2$. Note that $\lfloor \gamma_3(p)\rfloor=1$ for all $p\in Y$. Furthermore, the function $Y\to\R$ defined by $p\mapsto \gamma_3(p)-1$ is positive, and hence bounded away from zero by the compactness of $Y$. Fix any point $p\in Y$; it follows there exists some sufficiently large integer $m$ with $\frac{\lfloor \gamma_m(p)\rfloor}{m}=:\omega>\frac{1}{3}$. Since $\vr{Y}{r}$ is cyclic we have $\frac{\lfloor \gamma_{im}(p)\rfloor}{im}\ge\omega$ for all $i$, and hence Lemma~\ref{lem:wf-as-limit} gives $\wf(\vr{Y}{r})\ge\omega>\frac{1}{3}$.
\end{proof}

\section{Conclusions}

Let $Y=\{(x,y)\in \R^2~|~(x/a)^2+y^2=1\}$ be an ellipse with $a>1$. When $Y$ has small eccentricity, meaning $a\le\sqrt{2}$, it remains to study the homotopy types of $\vrc{Y}{r}$ when $r>r_2=\frac{4\sqrt{3}a^2}{3a^2+1}$. In addition, the case when $Y$ is not of small eccentricity, meaning $a>\sqrt{2}$, also deserves study. In this latter case, it may be that $\vr{X}{r}$ is not a cyclic graph for $X\subseteq Y$, and therefore we will need to use alternative techniques.

\section{Acknowledgements}
We would like to thank the anonymous reviewers for their helpful suggestions, and we would like to thank Justin Curry, Chris Peterson, and Alexander Hulpke for helpful conversations. MA was supported by VILLUM FONDEN through the network for Experimental Mathematics in Number Theory, Operator Algebras, and Topology. HA was supported in part by Duke University and by the Institute for Mathematics and its Applications with funds provided by the National Science Foundation. The research of MA and HA was supported through the program ``Research in Pairs" by the Mathematisches Forschungsinstitut Oberwolfach in 2015. SR was supported by a PRUV Fellowship at Duke University.

\bibliographystyle{plain}
\bibliography{OnVietorisRipsComplexesOfEllipses}

\begin{thebibliography}{10}

\bibitem{Adamaszek2013}
Micha{\l} Adamaszek.
\newblock Clique complexes and graph powers.
\newblock {\em Israel Journal of Mathematics}, 196(1):295--319, 2013.

\bibitem{AA-VRS1}
Micha{\l} Adamaszek and Henry Adams.
\newblock The {V}ietoris--{R}ips complexes of a circle.
\newblock {\em Pacific Journal of Mathematics}, 290:1--40, 2017.

\bibitem{AAFPP-J}
Micha{\l} Adamaszek, Henry Adams, Florian Frick, Chris Peterson, and Corrine
  Previte-Johnson.
\newblock Nerve complexes of circular arcs.
\newblock {\em Discrete \& Computational Geometry}, 56:251--273, 2016.

\bibitem{AAM}
Micha{\l} Adamaszek, Henry Adams, and Francis Motta.
\newblock Random cyclic dynamical systems.
\newblock {\em Advances in Applied Mathematics}, 83:1--23, 2017.

\bibitem{Carlsson2009}
Gunnar Carlsson.
\newblock Topology and data.
\newblock {\em Bulletin of the American Mathematical Society}, 46(2):255--308,
  2009.

\bibitem{chazal2009gromov}
Fr{\'e}d{\'e}ric Chazal, David Cohen-Steiner, Leonidas~J Guibas, Facundo
  M{\'e}moli, and Steve~Y Oudot.
\newblock Gromov--{H}ausdorff stable signatures for shapes using persistence.
\newblock In {\em Computer Graphics Forum}, volume~28, pages 1393--1403, 2009.

\bibitem{ChazalCrawley-BoeveyDeSilva}
Fr{\'e}d{\'e}ric Chazal, William Crawley-Boevey, and Vin de~Silva.
\newblock The observable structure of persistence modules.
\newblock {\em Homology, Homotopy and Applications}, 18(2):247--267, 2016.

\bibitem{ChazalDeSilvaOudot2013}
Fr{\'e}d{\'e}ric Chazal, Vin de~Silva, and Steve Oudot.
\newblock Persistence stability for geometric complexes.
\newblock {\em Geometriae Dedicata}, 173:193--214, 2014.

\bibitem{ChazalOudot2008}
Fr{\'e}d{\'e}ric Chazal and Steve Oudot.
\newblock Towards persistence-based reconstruction in {E}uclidean spaces.
\newblock In {\em Proceedings of the 24th Annual Symposium on Computational
  Geometry}, pages 232--241. ACM, 2008.

\bibitem{EdelsbrunnerHarer}
Herbert Edelsbrunner and John~L Harer.
\newblock {\em Computational Topology: An Introduction}.
\newblock American Mathematical Society, Providence, 2010.

\bibitem{Gromov1987}
Mikhael Gromov.
\newblock Hyperbolic groups.
\newblock In Stephen~M Gersten, editor, {\em Essays in Group Theory}. Springer,
  1987.

\bibitem{Hatcher}
Allen Hatcher.
\newblock {\em Algebraic Topology}.
\newblock Cambridge University Press, Cambridge, 2002.

\bibitem{Hausmann1995}
Jean-Claude Hausmann.
\newblock On the {V}ietoris--{R}ips complexes and a cohomology theory for
  metric spaces.
\newblock {\em Annals of Mathematics Studies}, 138:175--188, 1995.

\bibitem{Kozlov}
Dmitry~N Kozlov.
\newblock {\em Combinatorial Algebraic Topology}, volume~21 of {\em Algorithms
  and Computation in Mathematics}.
\newblock Springer, 2008.

\bibitem{kruse1969joint}
RL~Kruse and JJ~Deely.
\newblock Joint continuity of monotonic functions.
\newblock {\em The American Mathematical Monthly}, 76(1):74--76, 1969.

\bibitem{Latschev2001}
Janko Latschev.
\newblock Vietoris--{R}ips complexes of metric spaces near a closed
  {R}iemannian manifold.
\newblock {\em Archiv der Mathematik}, 77(6):522--528, 2001.

\bibitem{Reddy}
Samadwara Reddy.
\newblock The {V}ietoris--{R}ips complexes of finite subsets of an ellipse of
  small eccentricity.
\newblock Bachelor's thesis, Duke University, April 2017.

\bibitem{Vietoris27}
Leopold Vietoris.
\newblock {{\"U}ber den h{\"o}heren Zusammenhang kompakter R{\"a}ume und eine
  Klasse von zusammenhangstreuen Abbildungen}.
\newblock {\em Mathematische Annalen}, 97(1):454--472, 1927.

\bibitem{Hocolim}
Volkmar Welker, G{\"u}nter~M Ziegler, and \v{Z}ivaljevi{\'c} Rade~T.
\newblock Homotopy colimits --- comparison lemmas for combinatorial
  applications.
\newblock {\em Journal f{\"u}r die Reine und Angewandte Mathematik (Crelles
  Journal)}, 509:117--149, 1999.

\end{thebibliography}

\end{document}